\documentclass[aos,preprint]{imsart}

\setattribute{journal}{name}{}

\RequirePackage[colorlinks,citecolor=blue,urlcolor=blue]{hyperref}
\RequirePackage{natbib}

\RequirePackage[OT1]{fontenc}
\RequirePackage{amsthm,amsmath,graphicx,enumerate}

\usepackage{amssymb,tabularx,multicol,multirow,booktabs}
\usepackage{mhequ}
\usepackage{relsize}
\usepackage[top=1in, bottom=1in, left=1in, right=1in]{geometry}
\usepackage[toc,page]{appendix}
\startlocaldefs
\numberwithin{equation}{section}
\theoremstyle{plain}

\endlocaldefs

\input binary_macro.tex

\def \bx{\mathbf{x}}

\def \bX{\mathbf{X}}

\def \be{\begin{equs}}
	\def \ee{\end{equs}}

\def \I{\mathbf{1}}

\def \E{\mathbb{E}}
\def \P{\mathbb{P}}

\def \fhat{\hat{f}}

\def \sumn{\sum_{i=1}^n}
\def \bmu{\mathbf{\mu}}
\def \R{\mathbb{R}}

\def \Z {\mathbb Z}

\def \bQ{\mathbf{Q}}

\def \sumn{\sum_{i=1}^n}
\def \bmu{\pmb{\mu}}
\def \R{\mathbb{R}}

\def\G {\mathbb{G}}

\begin{document}
	\begin{frontmatter}
		\title{On Testing for Parameters in Ising Models}

		\begin{aug}
			
			\author{\fnms{Rajarshi} \snm{Mukherjee}\thanksref{m1}\ead[label=e2]{rmukherj@hsph.harvard.edu}}
			\and
			\author{\fnms{Gourab} \snm{Ray}\thanksref{m2}\ead[label=e3]{gourab1987@gmail.com}}

			\affiliation{Harvard University\thanksmark{m1}  and University of Victoria \thanksmark{m2} }
			
			\address{Department of Biostatistics,\\
				655 Huntington Avenue, Boston, MA- 02138. \\
				\printead{e2}
			}

			\address{ University of Victoria\\
				Mathematics and Statistics, DTB A425.\\
				\printead{e3}}

			
		\end{aug}
	
	\begin{abstract} 
		We consider testing for the parameters of Ferromagnetic Ising models. While testing for the presence of possibly sparse magnetizations, we provide a general lower bound of minimax separation rates which yields sharp results in high temperature regimes. Our matching upper bounds are adaptive over both underlying dependence graph and temperature parameter. Moreover our results include the nearest neighbor model on lattices, the sparse  Erd\"{o}s-R\'{e}nyi random graphs, and regular rooted trees -- right up to the critical parameter in the high temperature regime. We also provide parallel results for the entire low temperature regime in nearest neighbor model on lattices -- however in the plus boundary pure phase. Our results for the nearest neighbor model crucially depends on finite volume analogues of correlation decay property for both high and low temperature regimes -- the derivation of which borrows crucial ideas from FK-percolation theory and might be of independent interest. Finally, we also derive lower bounds for estimation and testing rates in two parameter Ising models -- which turn out to be optimal according to several recent results in this area. 
	\end{abstract}
	\end{frontmatter}
	

	\section{Introduction} 
	Inference of parameters in a dependent system of observations, especially of parameters in a Markov Random Field (MRF), has received a lot of attention in recent times. The main activities in this regard can be broadly classified in two categories: (i) infer the strength of dependence and individual significance of vertices given a \textit{single} sample from the MRF \citep{bhattacharya2015inference,ghosal2018joint,daskalakis2018testing,chatterjee2007estimation}, and (i) infer underlying graphical/dependence sructure among individuals given \textit{multiple} i.i.d. samples from the MRF \citep{bresler2015efficiently,anandkumar2011high,wu2013learning,santhanam2012information}. Here we focus on the first theme of problems and provide some results in terms of limits of hypothesis testing for parameters in Ising Models. 
	
	More precisely, let $\bX=(X_1,\ldots,X_n)^\top\in \{\pm 1\}^n$ be a random vector with the joint distribution of $\bX$ given by an Ising model defined as:
	\be
	\P_{\beta, \bQ,\bmu}(\bX=\bx):=\frac{1}{Z(\mathbf{Q}, \mathbf{\bmu})}\exp{\left(\frac{\beta}{2}\bx^\top\mathbf{Q} \bx+\bmu^\top\bx\right)},\qquad \forall \bx \in \{\pm 1\}^n,
	\label{eqn:general_ising}
	\ee
	where $\mathbf{Q}$ is an $n \times n$ symmetric and hollow matrix, $\bmu:=(\mu_1,\ldots,\mu_n)^\top\in \mathbb{R}^{n}$ is an unknown parameter vector to be referred to as the external magnetization vector, $\beta\in \mathbb{R}$ is a real number usually referred to as the ``inverse temperature", and $Z(\beta,\mathbf{Q}, \mathbf{\bmu})$ is a normalizing constant.  
	It is clear that the pair $(\beta,\mathbf{Q})$ characterizes the dependence among the coordinates of $\bX$, and $X_i$'s are independent if $\beta\mathbf{Q}=\mathbf{0}_{n\times n}$.  The matrix $\mathbf{Q}$ will usually be associated with a certain sequence of simple labeled graphs $\mathbb{G}_n=(V_n,E_n)$ with vertex set $V_n=\{1,\dots,n\}$ and edge set $E_n \subseteq V_n\times V_n$ and corresponding $\mathbf{Q}=|V_n| G_n/2|E_n|$, where $G_n$ is the adjacency matrix of $\mathbb{G}_n$. 
	Under model (\ref{eqn:general_ising}) we will in most parts of the paper be interested in testing 
	\be
	H_0: \bmu=\mathbf{0} \quad {\rm vs} \quad H_1: \bmu \in \Xi(s,A), \label{eqn:hypo_sparse}
	\ee
	where
	$${\Xi}(s,A):=\left\{\bmu\in \R^n:|{\rm supp}(\bmu)|= s,{\rm \ and\ } \min_{i\in {\rm supp}(\bmu)}\mu_i\geq A>0\right\},$$
	and
	$$
	{\rm supp}(\bmu):=\{1\le i\le n:\mu_i\ne 0\}.
	$$
In section \ref{section:aux_result} we shall also briefly discuss some results regarding the inference of both $\beta$ and $\bmu$  as well as its connections to some recent results in parameter estimation in Ising models \citep{chatterjee2007estimation,bhattacharya2015inference,ghosal2018joint} .
	
	To this end, we adopt a standard asymptotic minimax framework \citep{burnashev1979minimax, ingster1994minimax, Ingster1,Ingster4}. Let a statistical test for $H_0$ versus $H_1$ be a measurable $\{0,1\}$ valued function of the data $\bX$, with $1$ indicating rejecting the null hypothesis $H_0$ and $0$ otherwise. The worst case risk of a test $T: \{\pm 1\}^n\to \{0,1\}$ is then given by 
	\be
	\mathrm{Risk}(T,{\Xi}(s,A),\bQ)&:=\P_{\beta,\bQ,\mathbf{0}}\left(T(\bX)=1\right)+\sup_{\bmu \in {\Xi}(s,A)}\P_{\beta,\bQ,\bmu}\left(T(\bX)=0\right),\label{eqn:general_hypo_ising}
	\ee
	where $\P_{\beta,\mathbf{Q},\bmu}$ denotes the probability measure as specified by \eqref{eqn:general_ising}. We say that a sequence of tests $T$ corresponding to a the model-problem pair (\ref{eqn:general_ising}) and (\ref{eqn:general_hypo_ising}), to be asymptotically powerful (respectively asymptotically not powerful) against $\Xi(s,A)$ if
	\be
	\label{eqn:powerful}
	\limsup\limits_{n\rightarrow \infty}\mathrm{Risk}(T,{\Xi}(s,A),\bQ)= 0\text{ (respectively }\liminf\limits_{n\rightarrow \infty}\mathrm{Risk}(T,{\Xi}(s,A),\bQ)>0).
	\ee

	The minimum signal strength $A$ required (in sense of asymptotic rate) for guaranteeing the existence of asymptotically powerful tests will be refered to as the fundamental limits/minimax separation rates of the testing problem \eqref{eqn:hypo_sparse}. This problem has been previously considered by \cite{mukherjee2016global} where the authors study the effect of $\bQ=|V_n| G_n/2|E_n|$ for certain sequence of graphs $\mathbb{G}_n=(V_n,E_n)$ on the fundamental limits of the testing problem \eqref{eqn:hypo_sparse}. In particular, they show that 
	
	\begin{enumerate}\itemsep5pt
		\item[$\bullet$] For regular graphs with high enough diverging degree (large than $\sqrt{n}$) and $0\leq \beta<1$, the fundamental limits of testing \eqref{eqn:hypo_sparse} in the Ising model \eqref{eqn:general_ising} is identical (in rate) to the case of $\beta=0$. Moreover, the matching upper bounds crucially depended on the knowledge of both $\beta$ and $\bQ$ and hence is not adaptive over these nuisance parameters.
		
		\item[$\bullet$] The presence of a phase transition in the Ising model \eqref{eqn:general_ising} can reflect itself in the fundamental limits of testing  \eqref{eqn:hypo_sparse}. In particular, for Ising models on the cycle graph, the rate optimal results for all regimes of $\beta$ are the the same as the independent case (i.e. $\beta=0$). On the other hand, for the Curie-Weiss model (where $\bQ$ corresponds to the complete graph), whereas one can detect provably smaller rate optimal signals at the thermodynamic phase transition point  ($\beta=1$), rate optimality results matches the case of $\beta=0$ for any other $\beta>0$. 
		%
		%
		%
	\end{enumerate} 
	In view of the results stated above the following questions remain. 
	\begin{enumerate}\itemsep5pt
		\item [$\bullet$] For high temperature (small but non-vanishing $\beta>0$) in more general Ising models, does the testing problem \eqref{eqn:hypo_sparse} behave similar to $\beta=0$? Moreover, does there exist rate optimal testing procedures in such cases which do not depend on the knowledge of the nuisance parameters $\beta,\bQ$?
		
		\item [$\bullet$] What are the fundamental limits of testing \eqref{eqn:hypo_sparse} in regular graphs of bounded average degree (such as the classical nearest neighbor lattice ising model) for non-critical values of $\beta$ and whether they match the case of $\beta=0$?
		
		\item [$\bullet$] Does critical temperature in general help in detection i.e. is there a testing procedure which can detect lower order signals at a critical $\beta>0$ compared to what information theoretically possible under non-critical $\beta$?
	\end{enumerate}
	
	This paper is motivated in exploring the questions above. To proceed however, new proof techniques are necessary compared to those used in  \cite{mukherjee2016global}. In particular, the technical machinery employed in \cite{mukherjee2016global} is more tailored towards mean-field models. Consequently they are severely inadequate for getting sharp results right up to a critical point in bounded degree and/or sparse graphs, such as lattice Ising models and sparse Erd\"{o}s-R\'{e}nyi random graphs, due to the non-mean field nature of these problems. Finally, it is worth noting that obtaining the desired results for substantially small vanishing values of $\beta$ is  comparatively easier and the main thrust in this paper lies in pushing the answers right up to a critical $\beta$ whenever possible. Considering the motivations and challenges stated above, the main results of this paper can be summarized as below.
	
	\begin{enumerate}
		\item[(I)] We provide general matching upper and lower bounds under ``high temperature" regimes characterized by Dorbushin Type conditions. These yield optimal results up to critical temperature for regular rooted trees, sparse Erd\"{o}s-R\'{e}nyi Graphs, and nearest neighbor lattices.
		
		\item [(II)] In ``high temperature" regimes, our tests are free of both $\beta$ and the underlying graph -- which are nuisance parameters in the problem. Therefore, unlike \cite{mukherjee2016global} the tests are adaptive over these nuisance parameters.
		
		
		\item [(III)] The derivation of rate optimal results in nearest neighbor lattices in high as well as low temperature (in the pure phase) depends on correlation decay between spins in the Ising model. Although a limiting infinite volume version of the result is available \citep{aizenman1987phase,duminil2018exponential,duminil2017sharp}, we provide the finite $n$ versions of these results which require extra care -- especially for low-temperature regime.
		
		\item [(IV)] We provide some statistical physics based heuristics for the behavior of the problem at the critical temperature on nearest neighbor models on lattices. These results provide further intuition that at criticality it might be possible to detect lower signals.
		
		\item [(V)] Our results imply minimax optimality of pseudo-likelihood based estimator of the magnetization for many examples in classical two parameter Ising models \cite{ghosal2018joint}. To provide a more complete picture, we also provide a minimax lower bound for estimation of $\beta$ -- which matches the upper bound provided in \cite{ghosal2018joint} in many cases.
		
	\end{enumerate}
	

	\subsection{Organization} The rest of the paper is organized as follows. In Section \ref{section:general_graphs} we provide a general result stating matching upper and lower bounds to the testing problem \eqref{eqn:hypo_sparse} for regimes characterized by Dorbushin Type correlation decay condition. This section also collects a few classes of examples to demonstrate the applicability of the general result. Section \ref{section:latice} focuses on the special case of nearest neighbor lattices and contains optimal results in both high and low temperature (plus boundary pure phase) regimes. These results crucially depend on correlation decay results in these models -- the proof of which relies on FK-percolation theory. Consequently, Section \ref{section:technical_backgroud} is devoted to providing the necessary background. Section \ref{section:aux_result} contains some additional results on estimation-rate lower bounds for the classical two parameter Ising model on general graphs. Finally Section \ref{section:proofs_main} collects the proofs of all the results in the paper.

	\subsection{Notation}\label{section:preliminaries} 
	
	Throughout, for any probability measure $\P_{\theta}$, indexed by some paramter vector $\theta$, defined for the distribution of the random vector $\bX$ in study, we shall let we shall $\E_{\theta},\mathrm{Var}_{\theta},\mathrm{Cov}_{\theta}$ denote its expectation, variance and covariance operators respectively. For example,  $\E_{\beta,\bQ,\bmu},\mathrm{Var}_{\beta,\bQ,\bmu}$, $\mathrm{Cov}_{\beta,\bQ,\bmu}$ will denote the expectation, variance, and covariance operators corresponding to the measure $\P_{\beta,\bQ,\bmu}$. The coordinates of $X_i$ of $\bX$ following any Ising model will often be referred to as the spin for vertex $i$. For a given sequence of symmetric matrices $\mathcal{Q}=\{\bQ_{n\times n}\}_{n=2}^{\infty}$ (all with non-negative entries) we define the critical temperature as
	\be 
	\beta_{c}(\mathcal{Q})=\inf\left\{\beta>0: \lim_{h\downarrow 0}\lim_{n\rightarrow \infty}\E_{\beta,\bQ,\bmu(h)}\left(\frac{1}{n}\sum_{i=1}^n X_i\right)>0\right\},
	\ee
	where we let $\bmu(h)=(h,\ldots,h)^T\in \mathbb{R}^n$ denote the vector with all coordinates equal to $h$. Indeed, the definition above makes sense if the limits exist. In the examples pursued in the rest of the draft the existence of the limit is a par of classical statistical physics literature and we shall note relevant references whenever talking about critical temperature in our examples. We shall refer to low positive values of $\beta$ as high temperature and high positive values as low temperature regimes. The exact calibration of these low and high values will often be through $\beta_{c}(\mathcal{Q})$ -- once well defined (i.e. $0\leq \beta<\beta_{c}(\mathcal{Q})$ will correspond to high temperature and $\beta>\beta_{c}(\mathcal{Q})$ will be referred to as low temperature regime.) When referring to Ising models on sequence of graphs, $\mathcal{Q}$ will correspond to the sequence of scaled adjacency matrices of the graph.  
	
	We also let for any to Euclidean vectors $v_1,v_2$, $\|v_1-v_2\|_p$ denote the Euclidean $L_p$ norm between them.  We mostly use $\mathbb{G}_n=(V_n,E_n)$ to stand for simple labeled graph with vertex set $V_n=\{1,\dots,n\}$ and edge set $E_n \subseteq V_n\times V_n$, and denote its adjacency matrix by $G_n$. For any two vertices $i,j\in V_n$ of the graph $\mathbb{G}_n$, we shall use $d_{\mathbb{G}_n}(i,j)$ to denote the graph distance between $i$ and $j$ (and will often suppress the dependence of the notation on $\mathbb{G}_n$ once clear from the context). We shall also say that two vertices are connected in $\mathbb{G}_n$ and denote it by $i\stackrel{\mathbb{G}_n}{\leftrightarrow} j$ when there is a path in $\mathbb{G}_n$ connecting vertices $i$ and $j$. Finally, we let $\mathbb{Z}^d$ denote the integer lattice in $d$-dimensions. 
	
	We shall often abbreviate $\{1,\ldots,n\}$ as $[n]$ as well. The results in this paper are mostly asymptotic (in $n$) in nature and thus requires some standard asymptotic  notations.  If $a_n$ and $b_n$ are two sequences of real numbers then $a_n \gg b_n$ (and $a_n \ll b_n$) implies that ${a_n}/{b_n} \rightarrow \infty$ (and ${a_n}/{b_n} \rightarrow 0$) as $n \rightarrow \infty$, respectively. Similarly $a_n \gtrsim b_n$ (and $a_n \lesssim b_n$) implies that $\liminf_{n \rightarrow \infty} {{a_n}/{b_n}} = C$ for some $C \in (0,\infty]$ (and $\limsup_{n \rightarrow \infty} {{a_n}/{b_n}} =C$ for some $C \in [0,\infty)$). Alternatively, $a_n = o(b_n)$ will also imply $a_n \ll b_n$ and $a_n=O(b_n)$ will imply that $\limsup_{n \rightarrow \infty} \ a_n / b_n = C$ for some $C \in [0,\infty)$).

	\section{Main Results:}\label{section:main_result} We state the main results of the paper divided in two subsections. First, in Section \ref{section:general_graphs}, we show that under a Dorbushin type uniqueness condition on the null distribution $\P_{\beta,\bQ,\mathbf{0}}$ it is possible to get the matching upper and lower bounds for the testing problem \ref{eqn:hypo_sparse} in ferromagnetic Ising models. We subsequently provide a few examples of Ising models satisfying such conditions. Subsequently, in Section \ref{section:latice}, we concentrate on the special case of nearest neighbor model on lattices.

	\subsection{\bf General Graphs:} \label{section:general_graphs} We begin with a definition which will help us state the main result of this subsection.
	\begin{defn}\label{def_condition_dorbushin} Suppose $\bX\sim \P_{\beta,\bQ,\bmu}$. We say
		\begin{enumerate}
			\item  $(\beta,\bQ)$ is Ferromagnetic if $\min\limits_{i,j\in [n]}\beta \bQ_{i,j}\geq 0$.
			
			\item  $(\beta,\bQ)$ satisfies Condition (D) if there exists a $C>0$  such that for any $n\geq 1$ the following hold $$\|\bQ\|_{\infty \rightarrow \infty}\leq C\quad \text{and}\quad \max_{i\in [n]}\sum_{j=1}^n\mathrm{Cov}_{\beta,\bQ,\mathbf{0}}(X_i,X_j)\leq C.
			$$	
		\end{enumerate}

	\end{defn}
	
	The following is the main result of this subsection under the conditions introduced above.	
	
	\begin{theorem}\label{theorem:arbitrarysparse_general}
		Suppose $X\sim \P_{\beta,\bQ,\bmu}$  where $(\beta,\bQ)$ is Ferromagnetic satisfying condition (D). Then the following hold.
		\begin{enumerate}
			\item If $\tanh(A)\lesssim \sqrt{n}/s$ then no test is asymptotically powerful.
			\item If $\tanh(A)\gg \sqrt{n}/s$ then a test based on rejecting for large values of $\sum_{i=1}^n X_i$ is asymptotically powerful.
		\end{enumerate}
	\end{theorem}
	A few comments are in order regarding the conditions and implications of Theorem \ref{theorem:arbitrarysparse_general} as well as its connections to the results in \cite{mukherjee2016global}. We list them below.
	
	\begin{enumerate}\itemsep5pt
		\item [(i)] The Condition (D) is related to the decay of correlations between two sites $i$ and $j$ under null Ising measure $\P_{\beta,\bQ,\mathbf{0}}$ and not under the alternative measure $\P_{\beta,\bQ,\bmu}$. Consequently, such a condition is potentially ``easier" to verify. The nomenclature of condition (D) is used to reflect its similarity to classical Dorbushin's Condition of uniqueness of infinite volume Ising measures \citep[Chapter 6]{friedli2017statistical}. We note that the covariance part of Condition (D) is trivially true for $\beta=0$ and is expected to hold for small enough value of $\beta$ -- which will be referred to as high temperature regions following physics terminology. 
		
		\item [(ii)] The verification of Condition (D) will indeed follow from suitable control on pairwise correlations between $X_i$ and $X_j$. For example, for the mean-field Curie-Weiss model ($\bQ_{i,j}=\frac{1}{n}\mathbf{1}(i\neq j)$) it can be shown that (see Proposition \ref{prop_examples} below) $\mathrm{Cov}_{\beta,\bQ,\mathbf{0}}\leq \frac{C}{n}$ for $0\leq \beta<1$ and consequently the verification of Condition (D) is immediate in the high temperature region up to the critical temperature of $\beta=1$. We note that, Condition (D) is not true for the Curie-Weiss model for low temperature $\beta>1$ and a separate argument, not implied by Theorem \ref{theorem:arbitrarysparse_general}, is needed to tackle this. Indeed, it can be shown that \citep{comets1991asymptotics,mukherjee2016global}, Condition (D) holds conditional on $\sum_{i=1}^nX_i>0$) and this turns out to be the sufficient for the validity of the conclusion of Theorem \ref{theorem:arbitrarysparse_general}. For other Ising models, verification of Condition (D) under weakest possible conditions on $\beta$ is itself a research program and consequently in Section \ref{section:latice} we consider the special case of the nearest neighbor Ising model on lattices in $d$-dimensions (see Section \ref{section:latice}) in detail to explore verification of this condition. Similar to the Curie-Weiss model, we show that there exists a $\beta_c(d)$ (which coincides with a point of phase transition in nearest neighbor Ising models on $\mathbb{Z}^d$) such that for all $0\leq\beta<\beta_c(d)$ one has the validity of Condition (D). In contrast, for $\beta>\beta_c(d)$ we only show the validity of Condition (D) for the model with plus boundary condition (in essence similar to the $\sum_{i=1}^nX_i>0$ conditional statement valid for the Curie-Weiss Model).

		\item [(iii)] The if part of Theorem \ref{theorem:arbitrarysparse_general} only requires $\mathrm{Var}_{\beta,\bQ,\mathbf{0}}(\sum_{i=1}^n X_i)=O(n)$ which follows from  the second part of Condition (D).
		In contrast, 
		the full force of both the Ferro-magnetism and Condition (D) assumed on the pair $(\beta,\bQ)$ is required to prove the necessity part of the theorem. In this regard, the conditions and proof of a similar lower bound in \cite[Theorem 6]{mukherjee2016global} pertains to mean-field type graphs and especially can only handle regular graphs with degree diverging faster that $\sqrt{n}$. 
		Moreover, as we discuss in Section \ref{section:latice} and Section \ref{section:technical_backgroud}, the verification of correlation condition can benefit from several sharp results in probability theory and statistical physics. 
		\item [(iv)]	Note that our tests, when optimal, are free of both $\beta$ and the underlying graph (i.e. $\bQ$) -- which are nuisance parameters (potentially high dimensional) in the problem. In particular, the optimal test in Theorem \ref{theorem:arbitrarysparse_general} rejects whenever $\sum_{i=1}^nX_i\geq L_n\sqrt{n}$ for some slowly growing sequence $L_n$ which does not depend on with $\beta$ or $\bQ$. Therefore, unlike \cite{mukherjee2016global} the tests are adaptive over these nuisance parameters.
		
	\end{enumerate}

	Our next result verifies Condition (D) in specific classes of examples for high temperature regimes.
	
	\begin{prop}\label{prop_examples} Let $\mathbb{G}_n=(V_n,E_n)$ be a sequence of simple labeled graph with vertex set $V_n=\{1,\dots,n\}$ and edge set $E_n \subseteq V_n\times V_n$ and corresponding $\mathbf{Q}=|V_n| G_n/2|E_n|$, where $G_n$ is the adjacency matrix of $\mathbb{G}_n$. Then Condition (D) is satisfied whenever
		\begin{enumerate}
			
			\item $\mathbb{G}_n$ satisfies $|E_n|=\Theta(n^2)$ and $0\leq \beta<c$ where $c=\liminf_{n\rightarrow \infty}|E_n|/n^2$.
			
			
			\item $\mathbb{G}_n$ corresponds to a bounded degree graph of degree at most $k$ and $0\leq(k-1)^2\tanh(\beta/(k-1))<1$.
			
			\item $\mathbb{G}_n$ corresponds to a rooted regular tree of degree $k$ and $0\leq (k-1)\tanh(\beta/k)<1$.
		\end{enumerate}
		
	\end{prop}
	
	It is worth noting that none of these three classes of examples are covered by \cite[Theorem 6]{mukherjee2016global}. Especially, the proof of \cite[Theorem 6]{mukherjee2016global} relied heavily relied on both high degree  regularity of the underlying graph as well high temperature regime of $\beta$. In contrast, Proposition \ref{prop_examples} and Theorem \ref{theorem:arbitrarysparse_general} only rely on the high temperature requirement -- implied by the second part of Condition (D). 

	Unlike \cite{mukherjee2016global} however, the result for dense graphs ($|E|=\Theta(n^2)$) and bounded degree graphs above, does not guarantee validity of Condition (D) right up to the critical temperature in these models. Indeed, the critical temperature $\beta$ in these models are not immediately clear. Only when $|E_n|=n(n-1)$ and $\mathbb{G}_n$ corresponds to the Curie-Weiss model, the result above is valid up to the critical point in the model i.e. $\beta=1$. In contrast, for more general dense graphs converging in the cut-metric, the critical $\beta$ corresponds to the inverse of the Hilbert-Schmidt operator norm of the limiting graphon. We do not assume any such convergence of the underlying graph sequence since even under such convergence it is not immediately clear whether the results will hold right up to the critical temperature. 
	
	The result for rooted $k$-regular tree is sharp up to the thermodynamic phase transition of the model where $\beta_c$ satisfies $(k-1)\tanh(\beta_c/k)=1$ and recovers the case of the line graph where the result in Proposition \ref{prop_examples} holds for all $0\leq \beta<\infty$.  The result however is not true for Ising model on locally tree like graphs \citep{dembo2010ising}. Our next result derives detection thresholds for a special class of locally tree like graphs, namely, the sparse Erd\"{o}s-R\'{e}nyi Graphs model. In particular, we let  $\mathbb{G}_n=(V_n,E_n)\sim \mathcal{G}(n,\lambda/n)$ denote the Erd\"{o}s-R\'{e}nyi graph formed by joining each pair $i,j\in \{1,\ldots,n\}$ independently with probability $\frac{\lambda}{n}$ for some fixed $\lambda>0$. 
	
	\begin{prop}\label{prop_erdos_renyi}
		Let $\mathbb{G}_n=(V_n,E_n)\sim \mathcal{G}(n,\lambda/n)$ with adjacency matrix $G_n$ and  $\mathbf{Q}= G_n/\lambda$. Suppose $X\sim \P_{\beta,\bQ,\bmu}$. 
		\begin{enumerate}
			\item Suppose $0\leq \beta<\tanh^{-1}(1/2\lambda)$. Then asymptotically powerful tests for testing \eqref{eqn:hypo_sparse} exists if and only if $\tanh(A)\gg \sqrt{n}/s$.
			
			\item Irrespective of $\beta\in \mathbb{R}$, if $s\ll n$, no asymptotically powerful test exists if $\tanh(A)\ll \sqrt{n}/s$.
		\end{enumerate}
	\end{prop}
	
	Proposition \ref{prop_erdos_renyi} shows that for sparse  Erd\"{o}s-R\'{e}nyi Graphs, the lower bound of detection for any $\beta\in \mathbb{R}$ is the same as $\beta=0$. As we see in the proof, that this is a simple fact due to the existence of $O(n)$ isolated vertices in the graph. Therefore, the original claim in \cite{mukherjee2016global} (i.e. it is possible to detect lower in asymptotic order signals at critical temperature) needs to be somewhat modified to only considered connected graphs. In regard to the proof of upper bound in Proposition \ref{prop_erdos_renyi}, we first note that the condition $\|\mathbf{Q}\|_{\infty\rightarrow \infty}$ does not hold since maximum degree of a $\mathbb{G}_n\sim \mathcal{G}(n,\lambda/n)$ behaves asymptotically like $\log{n}/\log\log{n}$. The existence of such high degree vertices compared to the average degree of the graph, also destroys the validity of Condition (D). Consequently, the proof, although very simple, does not follow from either \cite{mukherjee2016global} or by checking the conditions of Theorem \ref{theorem:arbitrarysparse_general}. Finally it is well known \citep{dembo2010ising} that Ising Models on Erd\"{o}s-R\'{e}nyi graphs $\mathcal{G}(n,\lambda/n)$ undergo a phase transition with critical temperature $\beta_c(\mathcal{Q})=\tanh^{-1}(1/2\lambda)$ and we note that the upper bound holds right up to the critical point in the model i.e. $\tanh^{-1}(1/2\lambda)$ (here to put ourselves in the context of the definition of critical temperature in Section \ref{section:preliminaries} we take $\mathcal{Q}=\{\mathbf{Q}=G_n/2\lambda\}_{n\geq 1}$).

	Finally, going back to Proposition \ref{prop_examples}, we note that the requirement $0<(k-1)^2\tanh(\beta/(k-1))<1$ is not necessarily sharp for bounded degree graphs. The bound can be strengthened to $0<(k-1)\tanh(\beta/(k-1))<1$ for graphs of polynomial neighborhood growth. However, even then, this result is not sharp for nearest neighbor models on lattices. Subsequently, our next subsection is devoted to explore the case of nearest neighbor Ising models and understand the effect of critical $\beta$ on the testing problem \eqref{eqn:hypo_sparse}.

	\subsection{\bf Nearest Neighbor Interactions:}\label{section:latice}
	The main result of this section pertains to the detection thresholds for nearest neighbor Ising Models on lattices -- \textit{in any fixed dimension $d$}. We need a few notation to study the nearest neighbor Ising models. In particular, it is convenient to consider the points $i=1,\dots, n$ to be vertices of $d$-dimensional hyper-cubic lattice and the underlying graph (i.e. $\bQ$) to be the nearest neighbor (in sense of Euclidean distance) graph on these vertices. More precisely, given integers positive integers $n,d$, we consider a growing sequence of integer lattice hyper-cubes of dimension $d$ as $\Lambda_{n,d}=[-n^{1/d},n^{1/d}]^d\cap \mathbb{Z}^d$ where $\mathbb{Z}^d$ denotes the d-dimensional integer lattice. The elements of $\Lambda_{n,d}$ (and also those of $\mathbb{Z}^d$) will be referred to as vertices. Consider a family of random variables defined on the vertices of $\Lambda_{n,d}$ as $\bX \in \{-1,+1\}^{\Lambda_{n,d}}$ (to be referred to as ``spins" hereafter) having the following probability mass function (p.m.f.)
	\be
	\P_{\beta,\bQ,\bmu}(\bX=\bx)=\frac{1}{Z(\mathbf{Q}, \mathbf{\bmu})}\exp{\left(\frac{\beta}{2}\bx^\top\mathbf{Q} \bx+\bmu^\top\bx\right)},\qquad \forall \bx \in \{\pm 1\}^{\Lambda_{n,d}},
	\label{eqn:lattice_ising}
	\ee
	where as usual $\mathbf{Q}=(\bQ_{ij})_{i,j\in\Lambda_{n,d}}$ is a symmetric and hollow \textcolor{black}{array} (i.e. $\bQ_{ii}=0$ for all $i\in \Lambda_{n,d}$) with elements indexed by pairs of vertices in $\Lambda_{n,d}$ (organized in some pr-fixed lexicographic order), $\bmu:=(\mu_i:i\in \Lambda_{n,d})^\top\in \mathbb{R}^{\Lambda_{n,d}}$ referred to as the external magnetization vector indexed by vertices of $\Lambda_{n,d}$ , $\beta>0$ is the ``inverse temperature", and $Z(\beta,\mathbf{Q}(\Lambda_{n,d}), \mathbf{\bmu})$ is a normalizing constant. Note that, the in this notation $i,j\in \Lambda_{n,d}$ are vertices of the $d$-dimensional integer lattice and hence correspond to $d$-dimensional vector with integer coordinates. Therefore, using this notation, by nearest neighbor graph we shall mean $\bQ_{ij}=\bQ_{ij}(\Lambda_{n,d})=\mathbf{1}(0<\|i-j\|_1= 1)$.
	 and the main focus of this section  {will be to understand the detection thresholds (for testing $\bmu$) at both high,low, and critica dependence parameter $\beta$. To put us in the notation of critical temperature defined in Section \ref{section:preliminaries} we take $\mathcal{Q}(d)$ be the sequence of matrices $\{\bQ(\Lambda_{n,d})\}_{n\geq 1}$ and define
		\be 
		\beta_c(d)=\beta_{c}(\mathcal{Q}(d))=\inf\left\{\beta>0: \lim_{h\downarrow 0}\lim_{n\rightarrow \infty}\E_{\beta,\bQ(\Lambda_{n,d}),\bmu(h)}\left(\frac{1}{n}\sum_{i\in \Lambda_{n,d}} X_i\right)>0\right\},
		\ee
		where we let $\bmu(h)$ denote the vector in $\mathbb{R}^{|\Lambda_{n,d}|}$ with all coordinates equal to $h$.
		The existence and equivalence of the above notions (such as ones including uniqueness of infinite volume measure) of critical temperature in nearest neighbor Ising Model is a topic of classical statistical physics and we refer the interested reader to the excellent expositions in \cite{friedli2017statistical,duminil2017lectures} for more details.
		
		This value of $\beta_c(d)$ (which is known to be strictly positive for any fixed $d\geq 1$) is referred to as the critical inverse temperature in dimension $d$ and as mentioned in Section \ref{section:preliminaries} the behavior of the system of observations $\bX$ changes once $\beta$ exceeds this threshold. For $d=1$, it is known from the first work in this area \citep{ising1925beitrag} that $\beta_c(1)=+\infty$ and consequently the Ising model in 1-dimension is said to have no phase transitions. The seminal work of \cite{onsager1944crystal} provides a formula for $\beta_c(2)$ and obtaining an analytical formula for $\beta_c(d)$ for $d\geq 3$ remains open. Consequently, results only pertain to the existence of a strictly non-zero finite $\beta_c(d)$ which governs the macroscopic behavior of the system of observations $X_i,i\in \Lambda_{n,d}$ as $n\rightarrow \infty$. In particular, the average magnetization $n^{-1}\sum X_i$ converges to $0$ in probability for $\beta<\beta_c(d)$ and to a mixture of two delta-dirac random variables $m_+(\beta)$ and $m_-(\beta)=-m_+(\beta)$, for $\beta>\beta_c(d)$. This motivates defining Ising models in pure phases as follows. 
		Letting $\partial\Lambda_{n,d}$ denote the set of boundary vertices of $\Lambda_{n,d}$ w.r.t. the infinite $d$-dimensional lattice $\mathbb{Z}^d$, we denote \be
		\P^{+}_{\beta,\bQ(\Lambda_{n,d}),\bmu}(\bX=\bx)=\P_{\beta,\bQ(\Lambda_{n,d}),\bmu}(\bX=\bx|X_i=+1,i\in \partial\Lambda_{n,d}),\label{eqn:lattice_ising_plus}\\ \P^{-}_{\beta,\bQ(\Lambda_{n,d}),\bmu}(\bX=\bx)=\P_{\beta,\bQ(\Lambda_{n,d}),\bmu}(\bX=\bx|X_i=-1,i\in \partial\Lambda_{n,d}),\label{eqn:lattice_ising_minus}
		\ee
		to be Ising Models \eqref{eqn:general_ising} is $+$ and $-$ boundary conditions respectively. It is well known (\citep{ellis2007entropy}), that for $0\leq \beta<\beta_c(d)$ the asymptotic properties of the models $\P^{+}_{\beta,\bQ,\bmu}$, $\P^{-}_{\beta,\bQ,\bmu}$, and $\P_{\beta,\bQ,\bmu}$ (referred to as the Ising Model with free boundary conditions) are similar (i.e. they have all the same infinite volume $n\rightarrow \infty$ weak limit). However, for $\beta>\beta_c(d)$, the model $\P_{\beta,\bQ,\bmu}$ behaves asymptotically as the mixture of $\P^{+}_{\beta,\bQ,\bmu}$ and $\P^{-}_{\beta,\bQ,\bmu}$. 
		Indeed, for $\beta>\beta_c(d)$ it is not expected to have condition (D) hold for $\P_{\beta,\bQ,\bmu}$ and consequently we only present our result for the measure $\P^{+}_{\beta,\bQ,\bmu}$ in such cases. Although we provide intuitive reasoning (see Remark \ref{remark:negative_boundary} on why a similar result might hold for both negative boundary condition (i.e. $\P^{-}_{\beta,\bQ,\bmu}$) as well as free boundary condition (i.e. the original model $\P_{\beta,\bQ,\bmu}$) we do not yet have access to a rigorous argument in this regard. 
		
	\subsubsection{\bf Behaviour at Non-Critical  Temperatrure ($\beta\neq \beta_c(d)$)}\label{section:non_critical_lattice}

		Since we only proved Theorem \ref{theorem:arbitrarysparse_general} for the free boundary case, we state the following proposition which is proved in Section \ref{section:proofs_main}.
		
		\begin{prop}\label{prop:lattice_suff}
			Let $\beta>0$ and $\bQ_{ij}=\I(0<\|i-j\|_1\leq 1)$  for $i,j\in \Lambda_{n,d}$. 
			\begin{enumerate}
				\item Suppose $0\leq \beta<\beta_c(d)$ and $\mathrm{Cov}_{\beta,\bQ,\mathbf{0}}\left(X_i,X_j\right)\leq \exp\left(-c_{\beta,d}\|i-j\|_1\right)$ for some $c_{\beta,d}>0$ (depending on $\beta,d$)Then, for testing \eqref{eqn:hypo_sparse}, no tests are asymptotically powerful if $s\tanh(A)\lesssim \sqrt{n}$. 
				
				\item Suppose $\beta>\beta_c(d)$ and $\mathrm{Cov}^+_{\beta,\bQ,\mathbf{0}}\left(X_i,X_j\right)\leq c_{\beta,d}\exp\left(-c_{\beta,d}\|i-j\|_1\right)$ for some $c_{\beta,d}>0$ (depending on $\beta,d$) Then, for testing \eqref{eqn:hypo_sparse}, no tests are asymptotically powerful if $s\tanh(A)\lesssim \sqrt{n}$. 
			\end{enumerate}
		\end{prop}	
		
		The proof of the first part of this proposition is immediate by Theorem \ref{theorem:arbitrarysparse_general} (since the stated correlation decay and the nearest neighbor structure of $\bQ$ imply Condition (D)). The proof of the second part is partly similar to the proof of Theorem \ref{theorem:arbitrarysparse_general} with the difference being in the proof of \eqref{eq:null_large} part of the proof (see Section \ref{section:proofs_main}).
		
		To proceed, we simply note that in order to invoke Corollary \ref{prop:lattice_suff} for the nearest neighbor Ising Model, it is enough to check that the correlation between $X_i$ and $X_j$ decrease exponentially as a function of $\|i-j\|_1$. We state the exact result in our next theorem.
		%
		

		\begin{theorem}\label{thm:correlation_decay}
			Let $\beta>0$ and $\bQ_{ij}=\I(0<\|i-j\|_1\leq 1)$  for $i,j\in \Lambda_{n,d}$. Then there exists  $c_{\beta,d}>0$ (depending on $\beta$ and $d$) such that the following hold for every $n\geq 2$
			\be
			\mathrm{Cov}_{\beta,\bQ,\mathbf{0}}\left(X_i,X_j\right)&\leq \exp\left(-c_{\beta,d}\|i-j\|_1\right), \quad 0\leq \beta<\beta_c(d)\\
			\mathrm{Cov}^{+}_{\beta,\bQ,\mathbf{0}}\left(X_i,X_j\right)&\leq c_{\beta,d} \exp\left(-c_{\beta,d}\|i-j\|_1\right), \quad \beta>\beta_c(d).\label{eqn:correlation_decay}
			\ee
		\end{theorem}

		Theorem \ref{thm:correlation_decay} along with Proposition \ref{prop:lattice_suff} completes the proof of the claim that for Ising models on nearest neighbor lattices, away from critical temperature (at least in the positive pure phase) the detection thresholds for magnetization remains the same in terms on the necessary signal strength. The sufficiency of this signal strength follows easily by either the conditionally centered total magnetization test considered in \cite{mukherjee2016global} or simply the $\sum_{i=1}^n X_i$ based test (properly centered using $\E_{\beta,\bQ,\mathbf{0}}^+$ under the positive boundary condition case). Indeed, using a Peierls' argument type analysis (\cite{ellis2007entropy}) one can prove the result when $\beta$ is comfortably away from the critical $\beta_c(d)$. However, proving Theorem \ref{theorem:arbitrarysparse_general} for any $\beta\neq\beta_c(d)$ requires more refined analysis and recent correlation decay results of \cite{duminil2017sharp,duminil2018exponential} are crucial in this regard. However, applications of these recent results to our problem requires yet some more analysis. This is especially because \cite{duminil2017sharp,duminil2018exponential} works under the infinite volume weak limit of the model \eqref{eqn:lattice_ising_plus} and corresponding finite $n$ bounds (as we require) are not immediate. Deriving finite $n$ analogues of the correlation decay properties is involved and constitutes a major portion of our technical analysis. 
		In particular, we use standard results from FK-percolation theory obtained from the Edward-Sokal coupling of the Ising model with the Random Cluster Model (see \cite{grimmett2006random}) -- a background of which is provided in Section \ref{section:technical_backgroud}.
		
		\begin{remark}\label{remark:negative_boundary} We expect the detection thresholds presented in Proposition \ref{prop:lattice_suff} to also hold for the free boundary problem when $\beta>\beta_c(d).$ However, our proof crucially depends on the correlation decay property proved in Theorem \ref{thm:correlation_decay}. Although asymptotically a free boundary Ising model behaves like a mixture of positive and negative boundary ising models, we do not have GHS type inequality valid under negative boundary condition with positive magnetizations. Intuitively though a signal strength above the minimax separation rate should push the negative boundary condition towards a more positive boundary type model, we do not yet have a correct charactarization of this phenomenon and leave this to future efforts. 
			
		\end{remark}

		\subsubsection{\bf Behaviour at Critical  Temperatrure ($\beta=\beta_c(d)$)}\label{section:critical_lattice}
		
		Proposition~\ref{prop:lattice_suff} and Theorem~\ref{thm:correlation_decay} provide further evidence towards the philosophy put forward in  \cite{mukherjee2016global} -- i.e. even for non mean-field models (a substantially more challenging regime), away from critical temperature, the behavior of the detection problem \eqref{eqn:hypo_sparse} remains same as in the independent case (i.e. $\beta=0$). 
		The fact that at criticality ($\beta=\beta_c(d)$) one can detect lower signals (i.e. smaller in order $A$), is harder to prove rigorously. In this section, we provide a heuristic argument in a simpler subproblem which demonstrates this possible effect of criticality. To be more precise, we let $\bmu(h)=(\mu_i=h)_{i\in \Lambda_{n,d}}$ and consider the testing problem 
		\be 
		H_0: \bmu=\mathbf{0} \quad \text{vs} \quad H_1: \bmu=\bmu(h), h\geq A_n.
		\ee
		Note that this corresponds to the completely dense alternative $s=|\Lambda_{n,d}|$ in \eqref{eqn:hypo_sparse}. In this set up we provide statistical physics based heuristics below to demonstrate that a test based on rejecting $H_0$ for large values of $S_n=\sum_{i\in \Lambda_{n.d}}X_i$ can detect signals $A_n<<1/\sqrt{n}$ -- where $1/\sqrt{n}$ is shown to be a lower bound for testing at other temperatures in the previous subsection.
		
		Let us first consider the case $d \ge 4$. We claim that or any sequence $A_n$ such that  $A_n \to 0$ and $A_nn^{\frac{3(d-2)}{2d}} \to \infty$ such that the test based on rejecting for large values of $\sum_{i\in \Lambda_{n,d}}X_i$ is asymptotically powerful for testing \eqref{eqn:hypo_sparse} (note here that $\Xi(s,A)$ consists of a single element as $s  = n$ and $\mu_i \equiv h$). Observe that for $d \ge 4$, $\frac{3(d-2)}{2d} \in [3/4,3/2)$ which verifies the assertion that much smaller intensities are detectable at criticality.  The existence of such asymptotically powerful test for such a choice of $A_n$ immediately follows from the following (heuristic)  claim (and the fact that by GHS inequality (Lemma \ref{lemma:GHS}) $\Var_{\beta_c(d),\bQ(\Lambda_{n,d}),\boldsymbol 0)} (S_n   )\geq \Var_{\beta_c(d),\bQ(\Lambda_{n,d}),\bmu(h))} (S_n   )$ for any $h\geq 0$)
\begin{equation}
			\liminf_{n \to \infty} \frac{\E_{\beta_c(d),\bQ(\Lambda_{n,d}),\bmu(h)} (S_n  )}{\sqrt{\Var_{\beta_c(d),\bQ(\Lambda_{n,d}),\boldsymbol 0)} (S_n   )}}  = \infty. \label{ratio}
		\end{equation}
We now present a non-rigorous explanation for the above claim. The lower bound of the numerator follows from the following lower bound
\begin{equation}
\E_{\beta_c(d),\bQ(\Lambda_{n,d}),\bmu(A_n)} (S_n  ) \ge cnA_n^{1/3} \label{eq:numerator}
\end{equation}
for some constant $c>0$ independent of $n$.
This lower bound should in principle follow from the related lower bound for the mean of $X_0$ in the infinite volume setting, see Proposition 2.3 of \cite{duminil2016new} and also \cite{aizenman1987phase} (however we do not have a proof of this). On the other hand, it can be argued from classical infrared bounds (see for example Theorem 4.8 in  	\cite{duminil2017lectures}) that for some constant $c>0$
\begin{equation*}
\E_{\beta_c(d),\bQ(\Lambda_{n,d}),\boldsymbol 0} (X_i,X_j  )  \le \frac{c}{\|i-j\|^{d-2}}
\end{equation*}
from which it follows that $$\Var_{\beta_c(d),\bQ(\Lambda_{n,d}),\mathbf{0}} (S_n   )  \le cn^{1+2/d}.$$
Combining the above two results
$$
\liminf_{n \to \infty} \frac{\E_{\beta_c(d),\bQ(\Lambda_{n,d}),\bmu(h)} (S_n  )}{\sqrt{\Var_{\beta_c(d),\bQ(\Lambda_{n,d}),\boldsymbol 0} (S_n   )}}  \ge \liminf_{n \to \infty}  (A_n n^{\frac{3(d-2)}{2d}})^{1/3} =\infty
$$
by our choice of $A_n$.

 A rigorous proof of the above idea would require a finite volume analogue of the lower bound of the numerator \eqref{eq:numerator}. Ideally, one would want to extend Proposition 2.3 of \cite{duminil2016new} to the finite volume free boundary case. However, we believe this requires significant work which we do not pursue in this paper. It is worth mentioning here that with periodic boundary conditions (i.e. on a $d$-dimensional torus), it is an application of the work of \cite{AF86} that the numerator could be lower bounded by $cnA_n^{1/3}$. But unfortunately in that case, the required infrared bound for the denominator fails to hold. 
 
 \medskip

		We now turn to dimensions 2 and 3. For dimension 2, the discovery of Schramm Loewner Evolution (SLE) has sparked a series of works which provide us with rigorous proof of certain scaling exponents \cite{CDCHKS14,SLE}. For example, it is known that at criticality with zero external magnetic field, $S_n / n^{15/8}$ converges in law as $n \to \infty$ \citep{camia2015planar}. In dimension $3$, several exponents are conjectured to be true. Rather than surveying the literature on this, we present a non-rigorous discussion which explains why we believe that we can detect signals with intensity much lower than $n^{-1/2}$ at criticality in dimensions 2 and 3 as well and argue that the mean field case is the worst for detecting signals at criticality.  We refer to Section 3 of \cite{cardy} or  Chapter 9 of \cite{gri_perc} for more details. It is believed that the following two quantities should scale like
		$$
		\E_{\beta_c(d),\bQ(\Lambda_{n,d}),\bmu(h)} (X_0 ) \approx h^{1/\delta} \quad \quad  \quad \quad \sum_{j \in \Lambda_{n,d}}\E_{\beta_c(d),\bQ(\Lambda_{n,d}),\boldsymbol 0} (X_0 X_j)  \approx n^{\frac{2-\eta}{d}}
		$$ 
		Further, it is also believed that for $d \le d_c$ where $d_c$ is the upper critical dimension (the dimension above which all the exponents take their mean field values), there is an additional relation among these exponents called the \emph{hyperscaling relation}:
		$$
		2-\eta = d\frac{\delta -1}{\delta+1}.
		$$is approximately 
		\begin{equation*}
		\frac{ A_n^{1/\delta}n    }{  n^{\frac12 + \frac{2-\eta}{2d}}}  = A_n^{1/\delta} n^{\frac1{\delta+1}}.
		\end{equation*}
		Note that the right hand side diverges to infinity if $A_n \gg n^{-\frac{\delta}{\delta+1}} $. It is shown in \cite{aizenman1987phase} that $\delta$ satisfies the mean field lower bound $\delta \ge 3$. Using this we see that $n^{-\frac{\delta}{\delta+1} } \le n^{-3/4}$ and hence the mean field case is the worst for detecting signals at criticality.  This also hints that at criticality, one can detect lower values of signals for all $d \ge 2$ as opposed to the non-critical case.

		
		\section{Some Results for Two-Parameter Ising Models}\label{section:aux_result}
		In this section we collect results regarding inference in the classical two parameter Ising model -- which follow from similar proof ideas to those pursued in the rest of this paper. For a fixed pair of real numbers $(\beta,h)\in \mathbb{R}^+\times \mathbb{R}$ let $\bX\sim \P_{\beta,\bQ,\bmu(h)}$ where $\bmu(h)=(h,\ldots,h)\in \mathbb{R}^n$. Then \cite{ghosal2018joint} provide pseudo-likelihood estimates of $(\beta,h)$ and show joint $\sqrt{n}$-consistency of the estimates under certain classes of $\bQ$. Here we extend our previous results to show the optimality of their results in some of these situations. 
		
		\begin{theorem}\label{theorem:estimation}
			For any estimators $(\hat{\beta},\hat{h})$ of $(\beta,h)$ based on $\bX\sim \P_{\beta,\bQ,\bmu(h)}$ the following holds for has for any  $\beta^*$ and $0<\delta<\beta^*$ such that $\beta^*+\delta<(1-\rho)/\|\bQ\|_{\infty\rightarrow \infty}$ for some $0<\rho<1$.
			\be 
			\\ & sup_{\beta\in (\beta^*-\delta,\beta^*+\delta)\atop
				h\in \mathbb{R}}\E_{\beta,\bQ,\mathbf{0}}\left((\hat{\beta}-\beta)^2+(\hat{h}-h)^2\right)\\ 
			&\geq c\left(\frac{1}{\sum_{i,j=1}^n \bQ_{ij}^2}+\frac{1}{\mathrm{Var}_{\beta,\bQ,\bmu(0)}\left(\sum_{i=1}^n X_i\right)}\right)
			\ee
			for some constant $c>0$ depending on $\beta^*$.
			
			
			
		\end{theorem}
		We now comment on the result in Theorem \ref{theorem:estimation} in relation to recent results regarding inference in the two-parameter Ising Model $\P_{\beta,\bQ,\bmu(h)}$ \citep{ghosal2018joint,bhattacharya2015inference,chatterjee2007estimation}. 
		\begin{itemize}
			\item If  $\|\bQ\|_{\infty\rightarrow \infty}\leq C$ for some $C>0$ and 
			(precisely the condition (1.2)  of \cite{ghosal2018joint}) then we get a lower bound of mean squared error in estimation for jointly estimating $\beta,h$ as $\frac{1}{n}$. This is because in this case $\sum_{i,j=1}^n \bQ_{ij}^2\leq n\max_{ij}\bQ_{ij}^2\max_{i}\sum_{j=1}^n|\bQ_{ij}|\leq C^2n$. It also follows from \cite{ghosal2018joint} that if along with $\|\bQ\|_{\infty\rightarrow \infty}\leq C$ it further holds that  $\sum_{i,j=1}^n \bQ_{ij}^2\geq C'n$ for $C'>0$, then the pseudo-likelihood estimator of the pair $\beta,h$ has the usual parametric rate of convergence.  Therefore, in the case where $C''n\leq \sum_{i,j=1}^n \bQ_{ij}^2\leq C'n$ our result provides a matching lower bound for the result in  \cite[Theorem 1.5]{ghosal2018joint}. It is easy to check that this holds for Ising models on any connected bounded degree graphs.

			\item It is also established in \cite[Corollary 2.4]{bhattacharya2015inference} that an upper bound for the mean squared error rate of estimation of $\beta<1/\|\bQ\|_{\mathrm{op}}$ (where $\|\cdot\|_\mathrm{op}$ denotes the operator norm of a matrix) when $h=0$ is known is given by $1/\sum_{i,j=1}^n \bQ_{ij}^2$. Our Theorem matches this with a lower bound in a  sub-region (indeed by Gershgorin circle theorem
			 we have $1/\|\bQ\|_{\infty}\leq 1/\|\bQ\|_{\mathrm{op}}$).
			
			\item The term involving $\frac{1}{\mathrm{Var}_{\beta,\bQ,\bmu(0)}\left(\sum_{i=1}^n X_i\right)}$ in the statement of Theorem \ref{theorem:estimation} is also intuitive. In particular, $\mathrm{Var}_{\beta,\bQ,\bmu(0)}\left(\sum_{i=1}^n X_i\right)=C''n$ for some $C''>0$ when $\beta=0$ i.e. the independent spin case and in that case a $1/n$ mean squared error of convergence is obvious. Theorem \ref{theorem:estimation} extends this to general $\beta$. In particular, under condition (D) of Theorem \ref{theorem:arbitrarysparse_general}, since we have $\mathrm{Var}_{\beta,\bQ,\bmu(0)}\left(\sum_{i=1}^n X_i\right)\leq Cn$, we have that for these cases estimators of $h$ need to have mean squared error of order at least $1/n$. Consequently, in all the examples considered in Proposition \ref{prop_examples} the rate of convergence remains at least $1/n$.

		\end{itemize}
		
		
		

		
		\section{Technical Background}\label{section:technical_backgroud}
		In this section we discuss some relevant background on Ising models and collect some lemmas to be used in the proofs of the main results of this paper.

		\begin{lemma}[GHS Inequality \citep{lebowitz1974ghs}]\label{lemma:GHS}
			Suppose $X\sim \P_{\beta,\bQ,\bmu}$ with $\beta>0$, $\bQ_{ij}\geq 0$ for all $i,j$ and 
			$\bmu\in \left(\mathbb{R}^+\right)^{n}$. Then for any $(i_1,i_2,i_3)$ one has
			\be
			\frac{\partial^3\log{Z(\beta,\bQ,\bmu)}}{\partial \mu_{i_1}\partial \mu_{i_2}\partial \mu_{i_3}}\leq 0.
			\ee
			The inequality continues to hold for the measure $\P_{\beta,\bQ,\bmu}^+$ introduced in \eqref{eqn:lattice_ising_plus}.
		\end{lemma}
		
		We briefly discuss the significance of Lemma \ref{lemma:GHS}. First note that for any $i_1,i_2$, one has by standard theory of exponential family of distributions, 
		\be
		\frac{\partial^2\log{Z(\beta,\bQ,\bmu)}}{\partial \mu_{i_1}\partial \mu_{i_2}}=\mathrm{Cov}_{\beta,\bQ,\bmu}(X_i,X_j).
		\ee
		Consequently, Lemma \ref{lemma:GHS} implies that for any $\bmu_1\succcurlyeq\bmu_2\succcurlyeq\mathbf{0}$ (where $\succcurlyeq$ denotes coordinate-wise $\geq$ inequality between two vectors) one has
		\be
		\mathrm{Cov}_{\beta,\bQ,\bmu_1}(X_i,X_j)&\leq \mathrm{Cov}_{\beta,\bQ,\bmu_2}(X_i,X_j),\label{eqn:correlation_ordering}
		\ee
		whenever $\beta\bQ_{ij}\geq 0$ for all $i,j$.
		
		\begin{lemma}[GKS Inequality \citep{friedli2017statistical}]\label{lemma:GKS}
			Suppose $X\sim \P_{\beta,\bQ,\bmu}$ with $\beta>0$, $\bQ_{ij}\geq 0$ for all $i,j$ and $\bmu\in \left(\mathbb{R}^+\right)^{n}$. Then the following hold for any $i,j$
			\be
			\mathrm{Cov}_{\beta,\bQ,\bmu}(X_i,X_j)\geq 0; \quad \E_{\beta,\bQ,\bmu}(X_i)\geq 0.
			\ee
			The inequality continues to hold for the measure $\P_{\beta,\bQ,\bmu}^+$ introduced in \eqref{eqn:lattice_ising_plus}.
		\end{lemma}
		
		\begin{lemma}[Griffith's Second Inequality \citep{friedli2017statistical}]\label{lemma:griffith_second}
			Suppose $X^{(k)}\sim \P_{\beta^{(k)},\bQ^{(k)},\mathbf{0}}$ for  $k=1,2$ with $\beta^{(1)}\bQ\beta^{(1)}{ij}\geq \beta^{(2)}\bQ\beta^{(2)}{ij}\geq 0$ for all $i,j$. 
			Then
			\be
			\mathrm{Cov}_{\beta^{(1)},\bQ^{(1)},\mathbf{0}}(X_i,X_j)\geq \mathrm{Cov}_{\beta^{(2)},\bQ^{(2)},\mathbf{0}}(X_i,X_j), \quad \forall i,j.
			\ee
		\end{lemma}

		\begin{lemma}[FKG Inequality \citep{friedli2017statistical}]\label{lemma:FKG}
			Suppose $X\sim \P_{\beta,\bQ,\bmu}$ with $\beta>0$, $\bQ_{ij}\geq 0$ for all $i,j$ and $\bmu\in \left(\mathbb{R}^+\right)^{n}$. Suppose $f,g$ are two non-increasing functions of $\bx$ (i.e. whenever $\bx\succcurlyeq\bx'$ one has $f(\bx)\geq f(\bx')$ and $g(\bx)\geq g(\bx')$). Then 
			\be 
			\E_{\beta,\bQ,\bmu}\left(f(\bX)g(\bX)\right)\geq \E_{\beta,\bQ,\bmu}\left(f(\bX)\right)\E_{\beta,\bQ,\bmu}\left(g(\bX)\right).
			\ee
			The inequality continues to hold for the measure $\P_{\beta,\bQ,\bmu}^+$ introduced in \eqref{eqn:lattice_ising_plus}.
		\end{lemma}
		
		\begin{lemma}\label{lemma:expectation_vs_externalmag}
			Suppose $X\sim \P_{\beta,\bQ,\bmu}$  where $(\beta,\bQ)$ is Ferromagnetic satisfying condition (D) and $\bmu\in (\mathbb{R}^+)^n$ is such that $\|\bmu\|_{\infty}\leq M$ for some $M>0$. Then there exists $c_1,c_2>0$ (depending on $C,M$) such that
			
			\be
			c_1\mu_i\leq \E_{\beta,\bQ,\bmu}(X_i)\leq c_2 \mu_i+C\|\bmu\|_{\infty},\quad \forall i.
			\ee
			
		\end{lemma}
		
		\begin{proof}
			Note that
			\be
			\E_{\beta,\bQ,\bmu}(X_i)&= \E_{\beta,\bQ,\mathbf{0}}(X_i)+\sum_{j=1}^n \mu_j\mathrm{Cov}_{\beta,\bQ,\bmu^*}(X_i,X_j)
			\ee
			for some $\bmu^*$ lying on the line joining $\mathbf{0}$ and $\bmu$. Now note that  by GHS inequality (allowed by assumption 1 and $\bmu\in (\mathbb{R}^+)^n$)
			\be
			\sum_{j=1}^n \mu_j\mathrm{Cov}_{\beta,\bQ,\bmu^*}(X_i,X_j)&\leq \sum_{j=1}^n \mu_j\mathrm{Cov}_{\beta,\bQ,\mathbf{0}}(X_i,X_j)\\
			&\leq \mu_i\mathrm{Var}_{\beta,\bQ,\mathbf{0}}(X_i)+\max_{j\neq i}\mu_j\sum_{j\neq i}\mathrm{Cov}_{\beta,\bQ,\mathbf{0}}(X_i,X_j).
			\ee 
			
			Consequently the upper bound follows by assumption 2. For the lower bound note that we can safely assume $\E_{\beta,\bQ,\bmu}(X_i)\leq \frac{1}{2}$ because o.w. the result is true with $c=2M$. with note that by GKS inequality (allowed by assumption 1 and $\bmu\in (\mathbb{R}^+)^n$) we have $\E^2_{\beta,\bQ,\bmu^*}(X_i)\leq \E^2_{\beta,\bQ,\bmu}(X_i)$ and consequently
			\be
			\sum_{j=1}^n \mu_j\mathrm{Cov}_{\beta,\bQ,\bmu^*}(X_i,X_j)&\geq \mu_i\mathrm{Var}_{\beta,\bQ,\bmu^*}(X_i)\\
			&=\mu_i(1-\E^2_{\beta,\bQ,\bmu^*}(X_i))\geq \mu_i\left(1-\frac{1}{4}\right).
			\ee
			Therefore we have $\E_{\beta,\bQ,\bmu}\geq \min\{1/2M,3/4\}\mu_i$.	
		\end{proof}


		Our results also depend on the Edward-Sokal coupling between the Ising Model and the Random Cluster Model \cite{grimmett2006random}. To elaborate on this, consider $\bQ_{ij}=\I(0<\|i-j\|_1\leq L)$ as the adjacency matrix of a labelled undirected graph $\mathbb{G}_n$ with vertices $V(\mathbb{G}_n)=\Lambda_{n,d}$ and edges $E(\mathbb{G}_n)=\{(i,j):\bQ_{ij}\neq 0\}$. Note that $\mathbb{G}_n$ can be viewed as finite subgraph of finite range interaction graph of the lattice $\mathbb{Z}^d$ by considering $\bQ$ to be the restriction to $\Lambda_{n,d}$ of an infinite weighting matrix for pairs of vertices on the infinite lattice. Denote this infinite matrix as $\bQ(\mathbb{Z}^d)$ and the corresponding infinite graph by $\mathbb{G}=(V(\mathbb{G})=\mathbb{Z}^d,E(\mathbb{G}))$ where $E(\mathbb{G})=\{(i,j):\bQ_{ij}(\mathbb{Z}^d)\neq 0\}$. Also, for any two $i,j\in \mathbb{Z}^d$, let $d(i,j)$ be the graph distance between $i,j$ in $\mathbb{G}$ \textit{i.e.} $d(i,j)$ is the length of the shortest path between $i,j$ in $\mathbb{G}$. In case of nearest neighbor graph on $\mathbb{Z}^d$ one has $d(i,j)=\|i-j\|_1$. For any subgraph $H=(V(H),E(H))$ of $\mathbb{G}$ let $\partial H$ denote the set of all $i\in V(H)$ such that there exists $j\notin V(H)$ with $(i,j)\in E$. 
		
		With the notation above, a percolation configuration $\omega(H)=(\omega_{ij})_{(i,j)\in E(H)}$ over any subgraph $H=(V(H),E(H))$ of $\mathbb{G}$ is an element of $\{0,1\}^{E(H)}$. Any such percolation configuration $\omega(H)$ can be seen as a subgraph of $H$ with vertex set $V(H)$ and edge-set given by $\{(i,j)\in E(H): \omega_{ij}=1 \}$. Also, for any partition $\xi=P_1\cup \ldots P_l$ of $\partial H$, let $\omega^{\xi}(H)$ be obtained from $\omega(H)$ by contracting all vertices in each partitioning set $P_t,t=1,\ldots,m$ into one vertex and let $k(\omega^{\xi}(H))$ be the number of connected connected components of $\omega^{\xi}(H)$ (i.e. if two separated connected components in $\omega(H)$ include two separate vertices of a single $P_t,t=1,\ldots,l$, they are considered as a single component). Any such partition $\xi$ is referred to as a boundary condition and the Random Cluster Model on $H$ with $\xi$ boundary condition and parameters $(p,q)$ is a probability measure on percolation configurations $\omega(H)$ with p.m.f. $\phi^{\xi}_{p,q,H}$ defined as follows
		\be
		\phi_{p,q,H}^{\xi}(\omega(H))=\frac{p^{o(\omega(H))}(1-p)^{|E(H)|-o(\omega(H))}q^{k(\omega^{\xi}(H))}}{Z_{p,q,H}^{\xi}},
		\ee
		where $o(\omega(H))=\{(i,j)\in E(H): \omega_{ij}=1\}$ is referred to as the number of open edges of $\omega(H)$ and $Z_{p,q,H}^{\xi}$ is the normalizing constant. Incontext of the measure $\phi_{p,q,H}^{\xi}(\omega(H))$, we shall refer to $\xi=\cup_{l\in \partial H}\{\{l\}\}$ as the free boundary condition and $\xi=\{\partial H\}$ as the wired boundary condition. The following lemma collects some fundamental properties of the Random Cluster Model and connects it to the Ising Model \eqref{eqn:general_ising}. 
		
		\begin{lemma}[\cite{duminil2017lectures}]\label{lemma:random_cluster_model}
			Consider any subgraph $H=(V(H),E(H))$ of $\mathbb{G}$ and consider $\phi_{p,q,H}^{\xi}$ on $H$ with boundary condition $\xi$ and parameters $p>0$, $q\geq 1$. Let $\{i\stackrel{H}{\leftrightarrow}j\}$ denotes the event that $i,j$ are connected with edges in the random subgraph of $H$ obtained from $\omega(H)$.  
			\begin{enumerate}
				\item \textbf{Comparison Between Boundary Conditions:} Fix $i,j\in V(H)$ and suppose $\xi_1$ and $\xi_2$ are any two boundary conditions such that $\xi_1$ is a coarser partition than $\xi_2$. Then
				\be
				\phi_{p,2,H}^{\xi_2}(i\stackrel{H}{\leftrightarrow} j)&\leq \phi_{p,2,H}^{\xi_1}(i\stackrel{H}{\leftrightarrow} j).
				\ee 
				
				\item \textbf{Domain Markov Property:} Consider any subgraph $H'=(V(H'),E(H'))$ of $H$ and any boundary condition $\xi$ on $\partial H$. If $\psi'\in \{0,1\} ^{E(H')}$ and $\psi\in \{0,1\}^{E(H)\setminus E(H')}$ be any two configurations, then
				\be
				\phi^{\xi}_{p,q,H}\left(\omega(H)\vert_{E(H')}=\psi'\vert\omega(H)\vert_{E(H)\setminus E(H')}=\psi\right)&=\phi_{p,Q,H'}^{\psi^{\xi}}(\psi'),
				\ee
				where $\psi^{\xi}$ is the boundary condition on $\partial H'$ obtained by putting any two vertices in $\partial H'$ in the same partition when they are connected through the configuration $\psi$ and original boundary condition $\psi$. 
				\item \textbf{Edward-Sokal Coupling:} Suppose $\beta=-\frac{1}{2}\log{(1-p)}$ and let $\xi=\cup_{l\in \partial H}\{\{l\}\}$ and $\xi^{'}=\{\partial H\}$. Then 
				\be
			\ &	\phi_{p,2,H}^{\xi}(i\stackrel{H}{\leftrightarrow} j)=\E_{\beta,\bQ,\mathbf{0}}(X_iX_j),\\
				\ &  \phi_{p,2,H}^{\xi'}(i\stackrel{H}{\leftrightarrow} j)=\E_{\beta,\bQ,\mathbf{0}}^{+}(X_iX_j); \quad  \phi_{p,2,H}^{\xi'}(i\stackrel{H}{\leftrightarrow} \partial H)=\E_{\beta,\bQ,\mathbf{0}}^{+}(X_i).
				\ee
			\end{enumerate}
		\end{lemma}
		In the rest of the paper we shall work with $q=2$ (which corresponds to the Edward Sokal coupling of the Ising Model with the Random Cluster Model) and simply denote it by $\phi^{\xi}_{p,H}$ for boundary conditions $\xi$, subgraph $H$ (to be clear from context) and $p$ satisfying $\beta=-\frac{1}{2}\log{(1-p)}$.

		\section{Proofs of Main Results:}\label{section:proofs_main} 
		
		\subsection{Proof of Theorem \ref{theorem:arbitrarysparse_general}}

		Note that we can assume $s\gtrsim \sqrt{n}$ since otherwise the condition $s\tanh(A\gg \sqrt{n}$ cannot hold. To show that a test based on rejecting for large  values of $\sum_{i=1}^n X_i$. Note that by Condition (D), one has that $$\mathrm{Var}_{\beta,\bQ,\mathbf{0}}(\sum_{i=1}^n X_i)\leq Cn.$$ Consequently, the test that rejects whenever $\sum_{i=1}^n X_i\geq L_n\sqrt{n}$ for some sequence $L_n\rightarrow \infty$ has Type I error going to $0$. For the type II error, note that for any $\bmu\in (\mathbb{R}^+)^n$ one has by GHS inequality (Lemma \ref{lemma:GHS}) one has 
		$ 
		\mathrm{Var}_{\beta,\bQ,\bmu}(\sum_{i=1}^n X_i)\leq \mathrm{Var}_{\beta,\bQ,\mathbf{0}}(\sum_{i=1}^n X_i)\leq Cn. 
		$
		So it is enough to show that $\E_{\beta,\bQ,\bmu}(\sum_{i=1}^n X_i)\gg L_n\sqrt{n}$ for $\bmu\in \Xi(s,A)$ with $s\tanh(A)\gtrsim \sqrt{n}$. Fix $M>1$ and note that by Lemma \ref{lemma:expectation_vs_externalmag} one has with $\bmu^*(M)=(\min(\mu_i,M))$ that $$\E_{\beta,\bQ,\bmu}(\sum_{i=1}^n X_i)\geq \E_{\beta,\bQ,\bmu^*(M)}(\sum_{i=1}^n X_i)\gtrsim \sum_{i=1}^n \min(\mu_i,M) \gtrsim  s\tanh(A)\gg L_n\sqrt{n}$$ by choosing $L_n=\left(\sqrt{n}/s\tanh(A)\right)^{1/2}$ which diverges since $s\tanh(A)\gtrsim \sqrt{n}$. Note that in the last line we have used the fact that $\mu_i\geq A\geq \tanh(A)$ for nonzero $\mu_i$.

		It was also argued in \cite{mukherjee2016global} that this requirement on $A$ is rate optimal if the following holds for $\tanh(A)=c\sqrt{n}/s$ for some small constant $c>0$, any fixed $x>0$, and any sequence $x_n \to \infty$,

		\be
		\liminf_{n \to \infty}\P_{\beta,\bQ,\mathbf{0}} (  \sum_{i =1 }^n (X_i - \E_{\beta,\bQ,\mathbf{0}}(X_i) ) >x\sqrt{n} ) > 0,  \label{eq:null_large}\\
		\limsup_{n \to \infty} \sup_{\bmu\in\tilde\Xi(A,s)}  \P_{\beta,\bQ,\bmu} (  (\sum_{i =1 }^nX_i -  \E_{\beta,\bQ,\mathbf{0}}(X_i) ) ) >x_n \sqrt{n} ) = 0,  \label{eq:alt_small}
		\ee
		with $\tilde{\Xi}(s,A):=\left\{\bmu\in \mathbb{R}^{n}:|{\rm support}(\bmu)|= s,{\rm \ and\ } \min_{i\in {\rm support}(\bmu)}\mu_i= A>0\right\}$. Note that $\E_{\beta,\bQ,\mathbf{0}}(X_i)=0$ and we keep this in the statements to demonstrate the fact that the proof goes through even for the $+$-boundary problem considered in Section \ref{section:latice}.
		
		The first requirement \eqref{eq:null_large} pertains to exact $\sqrt{n}$-order of fluctuation of the total magnetization $\sum_{i=1}^nX_i$ and can be argued if a central limit theorem exists for the total magnetization $\sum_{i=1}^nX_i-\E_{\beta,\bQ,\mathbf{0}}(X_i)$. However, since we have a general Ising model, such a CLT is not available to us. \textcolor{black}{Instead we prove \eqref{eq:alt_small} first and derive \eqref{eq:null_large} using certain facts proved along the way.}
		
		Proving \eqref{eq:alt_small} can be subtle since it requires precise understanding of the fluctuations of the total magnetization under the alternative. To address this, we reduced our problem to a question about how the correlation between $X_i$ and $X_j$ under $H_0$ and conclude the proof using Condition (D).

		Note that it is enough to prove that for any $\bmu\in \tilde{\Xi}(s,A)$
		\be
		|\E_{\beta,\bQ,\bmu}(\sum_{i =1 }^nX_i)-\E_{\beta,\bQ,\mathbf{0}}\sum_{i =1 }^nX_i)|&\leq Cs\tanh(A),\label{eqn:non_centrality}\\
		\mathrm{Var}_{\beta,\bQ,\bmu}(\sum_{i =1 }^nX_i)&\leq Cn \label{eqn:variance}
		\ee
		for some constant $C>0$. We divide our proof in two parts depending on $s\gtrsim \sqrt{n}$ and $s\ll \sqrt{n}$.
		
		\subsubsection{Case 1: $s\gtrsim \sqrt{n}$} Since $s\gtrsim \sqrt{n}$ we can assume that any $\bmu\in \tilde{\Xi}(s,A)$ has $\|\bmu\|_{\infty}\leq M$ for some $M>0$ and consequently, $|\mu_j|\leq C\tanh(|\mu_j|)$ for some $C$ depending on $M$.

		To prove the first inequality \eqref{eqn:non_centrality}, we note that by property of exponential family of distributions and the mean value theorem 
		\be 
		\ &|\E_{\beta,\bQ,\bmu}(\sum_{i =1 }^nX_i)-\E_{\beta,\bQ,\mathbf{0}}(\sum_{i =1 }^nX_i)|\\&=|\sum_{i=1}^n \left(\frac{\partial}{\partial\mu_i}\log{Z(\beta,\bQ,\bmu)}-\frac{\partial}{\partial\mu_i}\log{Z(\beta,\bQ,\bmu)}\vert_{\bmu=\mathbf{0}}\right)|\\
		&=|\sum_{i=1}^n\sum_{j=1}^n\mu_j\frac{\partial^2}{\partial\mu_i\mu_j}\log{Z(\beta,\bQ,\bmu)}\vert_{\bmu=\bmu^*}|
		\ee
		where $\bmu^*$ lies on the line segment joining $\bmu$ and $\mathbf{0}$. Now, once again using properties of exponential family of models followed by GHS inequality (Lemma \ref{lemma:GHS}) we have
		\be 
		|\sum_{i=1}^n\sum_{j=1}^n\mu_j\frac{\partial^2}{\partial\mu_i\mu_j}\log{Z(\beta,\bQ,\bmu)}\vert_{\bmu=\bmu^*}|
		&=|\sum_{j=1}^n\mu_j\sum_{i=1}^n\mathrm{Cov}_{\beta,\bQ,\bmu^*}(X_i,X_j)|\\ &\leq |\sum_{j=1}^n\mu_j\sum_{i=1}^n\mathrm{Cov}_{\beta,\bQ,\mathbf{0}}(X_i,X_j)|\leq C\sum_{j=1}^n \mu_j
		\ee
		for some constant $C>0$ depending on $\beta$. The second last inequality uses the fact that covariances are positive (GKS Inequality in Lemma \ref{lemma:GKS}) and the fact that $\mu_j\geq 0$ for all $j\in \Lambda_{n,d}$. The last inequality is valid using Condition (D) and the fact that $\mu_j\geq 0$ for all $j\in \Lambda_{n,d}$. This concludes the verification of \ref{eq:alt_small}. 
		
		The verification of \eqref{eqn:variance} is also immediate since by GHS inequality $$\mathrm{Var}_{\beta,\bQ,\bmu}(\sum_{i =1 }^nX_i)=\sum_{i,j=1 }^n\mathrm{Cov}_{\beta,\bQ,\bmu}(X_i,X_i)\leq \sum_{i,j=1}^n\mathrm{Cov}_{\beta,\bQ,\mathbf{0}}(X_i,X_i)$$ and therefore the bound on the variance of $Cn$ for some constant $C>0$ depending on $\beta,d$ is obvious using Condition (D).

		It is clear that, by symmetry,
		\begin{equation}
		\label{eq:eq00}
		\P_{\beta,\mathbf{Q},\mathbf{0}}\Big(\Big|\sum_{i =1 }^n X_i\Big|>K\sqrt{n}\Big)=2\P_{\beta,\mathbf{Q},\mathbf{0}}\Big(\sum_{i =1 }^n X_i> K\sqrt{n}\Big).
		\end{equation}
		In establishing \eqref{eq:alt_small}, we essentially proved that
		\begin{equation}
		\label{eq:eq02}
		\limsup_{K\rightarrow\infty}\limsup_{n\rightarrow\infty}\sup_{\bmu\in \tilde{\Xi}(s,B)}\P_{\beta,\mathbf{Q},\bmu}\left(\sum_{i =1 }^nX_i>K\sqrt{n}\right)=0.
		\end{equation}
		By choosing $K$ large enough, we can make the right hand side of \eqref{eq:eq00} less than $1/2$. This gives 
		\be\label{eq:re-write}
		\sum_{{\bf x}\in \{-1,1\}^n}e^{\beta{\bf x}^\top\mathbf{Q}{\bf x}/2}\le 2\sum_{{\bf x}\in D_{n,K}}e^{\beta{\bf x}^\top\mathbf{Q}{\bf x}/2},
		\ee
		where  $D_{n,K}:=\Big\{|\sum_{i =1 }^n X_i|\le K\sqrt{n}\Big\}$. Then, setting $C_n:=\{\sum_{i =1 }^nX_i>\lambda \sqrt{n}\}$, for any $K>\lambda$ we have 
		\be
		\P_{\mathbf{Q},\mathbf{0}}(C_n)
		\ge \P_{\mathbf{Q},\mathbf{0}}(C_n\cap D_{n,K})
		=&\frac{\sum_{{\bf x}\in C_n\cap D_{n,K}} e^{\beta{\bf x}'\mathbf{Q}{\bf x}/2}}{\sum_{{\bf x}\in \{-1,1\}^n} e^{\beta{\bf x}'\mathbf{Q}{\bf x}/2}}\\
		\ge &\frac{1}{2}\frac{\sum_{{\bf x}\in C_n\cap D_{n,K}} e^{\beta{\bf x}'\mathbf{Q}{\bf x}/2}}{\sum_{{\bf x}\in D_{n,K}} e^{\beta{\bf x}'\mathbf{Q}{\bf x}/2}}\\
		\ge &\frac{e^{-2Kt}}{2}\frac{\sum_{{\bf x}\in C_n\cap D_{n,K}} e^{\beta{\bf x}'\mathbf{Q}{\bf x}/2+\frac{t}{\sqrt{n}}\sum_{i=1}^nx_i}}{\sum_{{\bf x}\in D_{n,K}} e^{\beta{\bf x}'\mathbf{Q}{\bf x}/2}}\\
		=&\frac{e^{-2Kt}}{2}\frac{\P_{\beta,\mathbf{Q},\bmu(t)}(C_n\cap D_{n,K})}{\P_{\beta,\mathbf{Q},\mathbf{0}}( D_{n,K})}\frac{Z(\beta,\mathbf{Q},\bmu(t))}{Z(\beta,\mathbf{Q},{\bf 0})}\\
		\ge &\frac{e^{-2Kt}}{2}\P_{\beta,\mathbf{Q},\bmu(t)}(C_n\cap D_{n,K}),
		\ee
		where $\bmu(t)=tn^{-1/2}{\bf 1}$. In the last inequality we use the fact that the function $t\mapsto Z(\beta,\mathbf{Q},\bmu(t))$ is non-increasing in $t$ on $[0,\infty)$, as
		$$\frac{\partial }{\partial t}Z(\beta,\mathbf{Q},\bmu(t))=\frac{1}{\sqrt{n}}\E_{\beta,\mathbf{Q},\bmu(t)}\sum_{i=1}^nX_i\ge \frac{1}{\sqrt{n}}\E_{\beta,\mathbf{Q},{\bf 0}}\sum_{i=1}^nX_i=0. $$
		To show \eqref{eq:null_large}, it thus suffices to show that there exists $K$ large enough and $t>0$ such that
		$$
		\liminf_{n\rightarrow\infty}\P_{\beta,\mathbf{Q},\bmu(t)}(C_n\cap D_{n,K})>0.$$
		To this end, it suffices to show that for any $\lambda>0$ there exists $t$ such that
		\begin{equation}
		\label{eq:eq01}
		\liminf_{n\rightarrow\infty}\P_{\beta,{\mathbf{Q}},\bmu(t)}(\sum_{i =1 }^nX_i>\lambda \sqrt{n})>0.
		\end{equation}
		If \eqref{eq:eq01} holds, then there exists $t>0$ such that 
		$$\liminf_{n\rightarrow\infty}\P_{\beta,\mathbf{Q},\bmu(t)}(C_n)>0.$$
		It now suffices to show that for any $t$ fixed one has
		\be
		\limsup_{K\rightarrow\infty}\limsup_{n\rightarrow\infty}\P_{\beta,\mathbf{Q},\bmu(t)}(D_{n,K}^c)=0,
		\ee
		which follows from \eqref{eq:eq02}.
		
		It now remains to show \eqref{eq:eq01}. To begin, note that for $h>0$ and $m_i(\bX)=\beta\sum_{j=1}^n\mathbf{Q}_{ij}X_j$
		\begin{eqnarray*}
			\E_{\beta,{\mathbf{Q}},\bmu(h)} X_i&=&\E_{\beta,{\mathbf{Q}},\bmu(h)} \tanh\left(m_i(\bX)+\frac{h}{\sqrt{n}}\right)\\
			&=&\E_{\beta,{\mathbf{Q}},\bmu(h)}\frac{\tanh(m_i(\bX))+\tanh\left(\frac{h}{\sqrt{n}}\right)}{1+\tanh(m_i(\bX))\tanh\left(\frac{h}{\sqrt{n}}\right)}\\
			&\ge& \frac{1}{2}\left[\E_{\beta,{\mathbf{Q}},\bmu(h)} \tanh(m_i(\bX))+\tanh\left(\frac{h}{\sqrt{n}}\right)\right]\\
			&\ge& \frac{1}{2}\tanh\left(\frac{h}{\sqrt{n}}\right).
		\end{eqnarray*}
		In the last inequality we use Holley inequality  \citep[e.g., Theorem 2.1 of][]{grimmett2006random} for the two probability measures $\P_{\beta,\mathbf{Q},\mathbf{0}}$ and $\P_{\beta,\mathbf{Q},\bmu(h)}$ to conclude $$\E_{\beta,\mathbf{Q},\bmu(h)}\tanh(m_i(\bX)\ge \E_{\beta,\mathbf{Q},\mathbf{0}}\tanh(m_i(\bX))=0,$$
		in the light of (2.7) of \cite{grimmett2006random}. Adding over $1\le i\le n$ gives
		\be\label{eq:mean_estimate}
		F_n'(h)=\frac{1}{\sqrt{n}}\E_{\beta,{\mathbf{Q},\bmu(h)} }\sum_{i =1 }^n X_i\ge \frac{\sqrt{n}}{2}\tanh\left(\frac{h}{\sqrt{n}}\right),
		\ee
		where $F_n(h)$ is the log normalizing constant for the model $\P_{\beta,\mathbf{Q},\bmu(h)}$. 
		Thus, using  Markov's inequality one gets
		\be
		\P_{\beta,\mathbf{Q},\bmu(t)}\left(\sum_{i =1 }^nX_i\le \lambda \sqrt{n}\right)
		=&\P_{\beta,\mathbf{Q},\bmu(t)}\left(e^{- \frac{1}{\sqrt{n}}\sum_{i =1 }^nX_i}\ge e^{- \lambda}\right)
		\\&\le \exp\left\{ \lambda+F_n(t- 1)-F_n(t)\right\},
		\ee
		Using \eqref{eq:mean_estimate}, the exponent in the rightmost hand side can be estimated as
		\be
		\lambda +F_n(t- 1)-F_n(t)
		= \lambda -\int_{t-1}^{t} F_n'(h)dh\le  \lambda-\frac{\sqrt{n}}{2}\tanh\left(\frac{t-1}{\sqrt{n}}\right),
		\ee
		which is negative and uniformly bounded away from $0$ for all $n$ large for $t=4\lambda+1$, from which \eqref{eq:eq01} follows.
		
		\subsubsection{Case 1: $s\ll \sqrt{n}$} 	In this case $s\tanh(A)\le C\sqrt{n}$ implies $s\le C'\sqrt{n}$, where $C':=C/\tanh(1)$. Also, since the statistic $\sum_{i=1}^nX_i$ is stochastically non-decreasing in $A$, without loss of generality it suffices to show that, for a fixed $S\subset[n]$ obeying $|S|=s$,
		\be\label{eqn:limit2}
		\limsup_{K\rightarrow\infty}\limsup_{n\rightarrow\infty}\limsup_{A\rightarrow\infty}\sup_{\bmu\in \tilde{\Xi}(s,A):\atop \mathrm{supp}(\bmu)=S}\P_{\beta,\mathbf{Q},\bmu}\left\{\sum_{i\in S^c}X_i>K\sqrt{n}\right\}=0.
		\ee
		Now, for $i\in S$ we have, using the fact that $\|\mathbf{Q}\|_{\ell_\infty\to\ell_\infty}\leq C$, for $\bmu \in \tilde{\Xi}(s,A)$
		\be
		\P_{\beta,\mathbf{Q},\bmu}(X_i=1|X_j&=x_j,j\ne i)=\frac{e^{A+m_i(\mathbf{x})}}{e^{A+m_i(\mathbf{x})}+e^{-A-m_i(\mathbf{x})}}\\
		&=\frac{1}{1+e^{-2m_i(\mathbf{x})-2A}}
		\ge \frac{1}{1+e^{2C-2A}},
		\ee
		and so $\lim_{A\rightarrow\infty}\P_{\beta,\mathbf{Q},\bmu}(X_i=1,i\in S)=1$ uniformly in $\bmu \in \tilde{\Xi}(s,A)$ with  $s\leq C'\sqrt{n}$.
		Also note that for any configuration $(x_j,j\in S^c)$ we have
		\be
		\P_{\beta,\mathbf{Q},\bmu}(X_i=x_i,i\in S^c|X_i=1,i\in S)&\propto \exp\left(\frac{\beta}{2}\sum_{i,j\in S^c}x_ix_j\bQ_{ij}+\sum_{i\in S^c} x_i\tilde{\mu}_{S,i}\right), \label{eqn:conditional_distribution}
		\ee
		where $\tilde{\mu}_{S,i}:=\sum_{j\in S}\bQ_{ij}\le \|\mathbf{Q}\|_{\ell_\infty\to\ell_\infty}$. Further we have
		\be
		\sum_{i=1}^n\tilde{\bmu}_{S,i}=\sum_{i=1}^n\sum_{j\in S}\bQ_{ij}=\sum_{j\in S}\sum_{i=1}^n\bQ_{ij}\le C'\|\mathbf{Q}\|_{\ell_\infty\to\ell_\infty} \sqrt{n}. \label{eqn:mu_tilde}
		\ee
		
		Thereafter repeating the argument of the case $s\gtrsim\sqrt{n}$ with the conditional measure \ref{eqn:conditional_distribution}. More precisely, since the conditional measure \ref{eqn:conditional_distribution} is an Ising model where the total external magnetization $\sum_{i=1}^n\tilde{\bmu}_{S,i}$ does not exceed $O(\sqrt{n})$ and by Griffith's Second Inequality (Lemma \ref{lemma:griffith_second})  and GHS Inequality (Lemma \ref{lemma:GHS}) the pairwise covariance between spins in the conditional measure \ref{eqn:conditional_distribution} is smaller than that in the actual measure $\P_{\beta,\bQ,\mathbf{0}}$ -- one has similar control over $\sum_{i\in S^c}X_i$ under the conditional measure as obtained before for $\sum_{i=1}^n X_i$ in the case $s\gtrsim\sqrt{n}$. This completes the proof.
		


		\subsection{Proof of Proposition \ref{prop_examples}}
		\begin{enumerate}
			\item Note that by Griffith's Second Inequality (Lemma \ref{lemma:griffith_second}) it is enough to prove the result for the complete graph i.e. $\bQ_{ij}=\frac{1}{n}\mathbf{1}(i\neq j)$. In this regard, 	using\cite[lemma 3]{mukherjee2016global} we know that there exists a random variable $T_n\sim N(\bar{X},1/n)$ such that given $T_n=t_n$ we have $(X_1,\cdots,X_n)$ i.i.d. with
			$$\P_{\beta,\bQ,\mathbf{0}}(X_i=x_i|T_n=t_n)=\frac{e^{\beta t_n x_i}}{e^{\beta t_n}+e^{-\beta t_n}}.$$
			This immediately gives
			$$\E_{\beta,\bQ,\mathbf{0}}(X_iX_j)=\E_{\beta,0}\tanh(\beta T_n)^2\le \beta^2\E_{\beta,\bQ,\mathbf{0}}(T_n^2).$$
			Also using the proof of \cite[lemma 3]{mukherjee2016global} we have that $T_n$ has density proportional to $e^{-nf(t)}$, where one has for $0<\beta<1$ that $$\lambda_1 \frac{t^2}{2}\le f(t)=\beta \frac{t^2}{2}-\log\cosh(\beta t)\le \lambda_2\frac{t^2}{2}$$
			for constants $\lambda_1,\lambda_2$ depending on $\beta$.
			This gives
			\begin{align*}
			\E_{\beta,\bQ,\mathbf{0}}T_n^2\le \frac{\int_\R t^2 e^{-n\lambda_1 t^2/2}dt}{\int_\R e^{-n\lambda_2 t^2/2}dt}=\frac{(\lambda_2)^{1/2}}{n(\lambda_1)^{3/2}},
			\end{align*}
			from which it follows that
			\be 
			\mathrm{Cov}_{\beta,\bQ,\mathbf{0}}(X_i,X_j) = \E_{\beta,\bQ,\mathbf{0}}(X_iX_j)\leq \beta^2\frac{(\lambda_2)^{1/2}}{n(\lambda_1)^{3/2}}.
			\ee
			Consequently, 
			\be 
			\max_{i}\sum_{j=1}^n\mathrm{Cov}_{\beta,\bQ,\mathbf{0}}(X_i,X_j)\leq \beta^2\frac{(\lambda_2)^{1/2}}{(\lambda_1)^{3/2}}.
			\ee
			This completes the proof.
			
			\item Note that (see e.g, \cite[equation 30]{anandkumar2011high}) $\mathrm{Cov}_{\beta,\bQ,\mathbf{0}}=\|\P_{\beta,\bQ,\mathbf{0}}(X_i|X_j=1)-\P_{\beta,\bQ,\mathbf{0}}(X_i|X_j=-1)\|_{TV}$ i.e. the total variational distance between the conditional distribution of $X_i|X_j=1$ and conditional distribution of $X_i|X_j=-1$. To bound the total variation distance note that we have from \cite[Lemma 3.4]{wu2013learning} we have the following. If the graph distance between $i,j$ is $l$ then for  $x_i\in \{-1,+1\}$
			\be 
			|\P_{\beta,\bQ,\mathbf{0}}(X_i=x_i|X_j=1)-\P_{\beta,\bQ,\mathbf{0}}(X_i=x_i|X_j=-1)|
			&\leq \gamma \delta^l
			\ee
			where $\gamma=\frac{4\beta}{(k-1)\tanh(\beta/k)}$ and $\delta=(k-1)\tanh(\beta/k)$. The result follows since the number of vertices at a distance $l$ is at most $k(k-1)^{l-1}\leq 2(k-1)^{l}$ for $k\geq 2$.
			
			\item Note that for a rooted regular tree with degree $k$, for any vertices $i,j$ such that the graph distance between $i,j$ equals $l$, $\mathrm{Cov}_{\beta,\bQ,\mathbf{0}}=\left(\tanh(\beta/k)\right)^l$. This can be easily proved using a recursive argument -- see e.g. \cite{anandkumar2011high}. The result follows since the number of vertices at a distance $l$ is at most $k(k-1)^{l-1}\leq 2(k-1)^{l}$ for $k\geq 2$.
			
		\end{enumerate}

		\subsection{Proof of Proposition \ref{prop_erdos_renyi}}
		
		\begin{enumerate}
			\item Note that we can assume $s\gtrsim \sqrt{n}$ since otherwise the condition $s\tanh(A\gg \sqrt{n}$ cannot hold. We show that a test based on rejecting for large  values of $\sum_{i=1}^n X_i$. First note that, even though Condition (D) does not hold here, we have by \cite{giardina2015quenched}  $$\mathrm{Var}_{\beta,\bQ,\mathbf{0}}(\sum_{i=1}^n X_i)\leq Cn$$ \textcolor{black}{with high probability w.r.t. the randomness of $\G_n$}.
			Consequently, the test that rejects whenever $\sum_{i=1}^n X_i\geq L_n\sqrt{n}$ for some sequence $L_n\rightarrow \infty$ has Type I error going to $0$. For the type II error, note that for any $\bmu\in (\mathbb{R}^+)^n$ one has by GHS inequality (Lemma \ref{lemma:GHS}) that
			$ 
			\mathrm{Var}_{\beta,\bQ,\bmu}(\sum_{i=1}^n X_i)\leq \mathrm{Var}_{\beta,\bQ,\mathbf{0}}(\sum_{i=1}^n X_i)\leq Cn. 
			$
			So it is enough to show that $\E_{\beta,\bQ,\bmu}(\sum_{i=1}^n X_i)\gg L_n\sqrt{n}$ for $\bmu\in \Xi(s,A)$ with $s\tanh(A)\gtrsim \sqrt{n}$. Fix $M>1$ and note that by Lemma \ref{lemma:expectation_vs_externalmag} one has with $\bmu^*(M)=(\min(\mu_i,M))$ that \be
			\E_{\beta,\bQ,\bmu}(\sum_{i=1}^n X_i)\geq \E_{\beta,\bQ,\bmu^*(M)}(\sum_{i=1}^n X_i)&\gtrsim \sum_{i=1}^n \min(\mu_i,M) \\
			&\gtrsim  s\tanh(A)\gg L_n\sqrt{n}
			\ee
			by choosing $L_n=\left(\sqrt{n}/s\tanh(A)\right)^{1/2}$ which diverges since $s\tanh(A)\gtrsim \sqrt{n}$.

			\item This follows trivially from noting that there exist a a constant $c>0$ (depending on $\lambda$) such that $\G_n$ has $\geq cn$ isolated vertices with probability converging to $1$ (see e.g. \cite[Chapter 3]{frieze2015introduction}). Consequently, when the $s$ signal magentizations are on those isolated vertices, the problem is as hard as detecting sparse signals for independent binary outcomes -- a basic case considered in \cite{mukherjee2015hypothesis}. The result therefore follows for any $\beta$.

		\end{enumerate}

		\subsection{Proof of Proposition \ref{prop:lattice_suff}}
	Following the proof of \cite[Theorem 6]{mukherjee2016global} we note that it is enough to show that
		
		\be
		\liminf_{n \to \infty}\P^+_{\beta,\bQ,\mathbf{0}} (  \sum_{i\in \Lambda_{n,d} }^n (X_i - \E^+_{\beta,\bQ,\mathbf{0}}(X_i) ) >x\sqrt{n} ) > 0,  \label{eq:null_large_plus}\\
		\limsup_{n \to \infty} \sup_{\bmu\in\tilde\Xi(A,s)}  \P^+_{\beta,\bQ,\bmu} (  \sum_{i\in \Lambda_{n,d} }(X_i -  \E^+_{\beta,\bQ,\mathbf{0}}(X_i) ) ) >x_n \sqrt{n} ) = 0,  \label{eq:alt_small_plus}
		\ee
		
		with $\tilde{\Xi}(s,A):=\left\{\bmu\in \mathbb{R}^{n}:|{\rm support}(\bmu)|= s,{\rm \ and\ } \min_{i\in {\rm support}(\bmu)}\mu_i= A>0\right\}$. \\

		Once again to show \eqref{eq:alt_small_plus} it is enough to show that for some constant $C>0$
		\be
		|\E^+_{\beta,\bQ,\bmu}(\sum_{i\in \Lambda_{n,d}}X_i)-\E^+_{\beta,\bQ,\mathbf{0}}(\sum_{i\in \Lambda_{n,d}}X_i)|&\leq Cs\tanh(A),\label{eqn:non_centrality_plus}\\
		\mathrm{Var}^+_{\beta,\bQ,\bmu}(\sum_{i\in \Lambda_{n,d} }X_i)&\leq Cn \label{eqn:variance_plus}
		\ee
		The proof of \eqref{eqn:non_centrality_plus} and \eqref{eqn:variance_plus} follows verbatim as the proof of \eqref{eqn:non_centrality} and \eqref{eqn:variance} in Theorem \ref{theorem:arbitrarysparse_general} -- since both GHS and GKS inequality continue to hold here along with the exponential correlation decay stated in statement of the proposition.
		
		The proof of \eqref{eqn:non_centrality_plus} is a by product of Central Limit Theorem under $H_0$. Indeed, the fact that exponential correlation decay $\mathrm{Cov}^{+}_{\beta,\bQ,\mathbf{0}}\left(X_i,X_j\right)\leq \exp\left(-c_{\beta,d}\|i-j\|_1\right)$ implies that $\sum_{i\in \Lambda_{n,d}} (X_i-\E_{\beta,\bQ,\mathbf{0}}^+(X_i))$ is asymptotically normal follows from arguments in \cite[Theorem 5]{martin1973mixing}.

		\subsection{Proof of Theorem \ref{thm:correlation_decay}}
		We recall the definition $\beta_c(d)$ from Section \ref{section:technical_backgroud} and divide our analysis in two parts based on the position of $\beta$ relative to $\beta_c(d)$. 
		
		\subsubsection{High temperature $(0<\beta<\beta_c(d))$} 
		
		First note that by GKS inequality (Lemma \ref{lemma:GKS}) one has 
		\be
		\mathrm{Cov}_{\beta,\bQ,\mathbf{0}}\left(X_i,X_j\right)&\leq \E_{\beta,\bQ,\mathbf{0}}(X_iX_j).
		\ee
		Subsequently, by Lemma \ref{lemma:random_cluster_model} part 3, one has for $p=p(\beta)=1-e^{-2\beta}$ and $\xi=\{\{l\}:l\in \partial G_n\}$
		\be
		\E_{\beta,\bQ,\mathbf{0}}(X_iX_j)&=\phi_{p,\G_n}^{\xi}(i\stackrel{\G_n}{\leftrightarrow} j).
		\ee
		Now let $\G_n'$ be the subgraph of $\mathbb{G}$ obtained by considering the ball (in graph distance induced by $\mathbb{G}$) of radius $n^{1/d}$ around $\mathbf{0}$. Then it is clear that $\G_n$ is subgraph of $\G_n'$. Let $\xi'=\{\{l\}:l\in \partial \G_n'\}$. However, one has for $i,j\in \Lambda_{n,d}$
		\be
		\phi_{p,\G_n}^{\xi}(i\stackrel{\G_n}{\leftrightarrow} j)&\leq \phi_{p,\G_n'}^{\xi'}(i\stackrel{\G_n'}{\leftrightarrow} j)\leq \phi_{p,\G_n'}^{\partial \G_n'}(i\stackrel{\G_n'}{\leftrightarrow} j).
		\ee
		Above the first inequality follows from Lemma \ref{lemma:random_cluster_model} part 2 followed by part 1 (the probability of $i$ is connected to $j$ in the larger $\G_n'$ is average over $i$ and $j$ conencted in the smaller graph $\G_n$ with random boundary conditions each of which dominates the boundary condition $\xi$), and then the second inequality uses Lemma \ref{lemma:random_cluster_model} part 1 again. Now let $R_n(i,j)$ be the ball of radius $d(i,j)$ around $i$ where $d(i,j)$ is the graph distance between vertices $i$ and $j$ (in this case simply equals $\|i-j\|_1$ for the nearest neighbor Ising model on lattices). Consequently $j\in \partial R_n(i,j)\subseteq G_n'$ and by Lemma \ref{lemma:random_cluster_model} part 2 followed by part 1
		\be
		\phi_{p,\G_n'}^{\partial \G_n'}(i\stackrel{\G_n'}{\leftrightarrow} j)&\leq \phi_{p,R_n(i,j)}^{\partial R_n(i,j)}(i\stackrel{R_n(i,j)}{\leftrightarrow} j)\leq \phi_{p,R_n(i,j)}^{\partial R_n(i,j)}(i\stackrel{R_n(i,j)}{\leftrightarrow} \partial R_n(i,j)).
		\ee  
		Finally, by \cite[Theorem 1.2]{duminil2017sharp}, there exists $\beta_c(d)>0$ such that for all $\beta<\beta_c(d)$
		\be
		\phi_{p,R_n(i,j)}^{\partial R_n(i,j)}(i\stackrel{R_n(i,j)}{\leftrightarrow} \partial R_n(i,j))&\leq \exp\left(-c_{\beta}d(i,j)\right)
		\ee
		for some $c_{\beta}>0$ depending on $\beta$. This completes the proof for the high temperature regime.

		
		\subsubsection{Low temperature $(\beta>\beta_c(d))$}
		We will use the main result of \cite{duminil2018exponential}. Observe that in \cite{duminil2018exponential}, truncated correlation decay was proved in an infinite volume setting, hence we need a bit of work to use that in the finite volume case.

		We begin with a few notation and known results and we cite \cite{duminil2018exponential} for all the relevant references. Given $p \in (0,1)$, let $\phi_{p}$ be the unique infinite volume measure obtained as the limit of a sequence of FK-Ising measures on a sequence of finite graphs which exhaust $\mathbb Z^d$. This measure is unique in the sense that the limit does not depend on the sequence taken, nor on the boundary conditions imposed on each of them. For any $j\in \mathbb{Z}^d$ we let $\Lambda_k(j)$ to denote the set $\{l\in\mathbb{Z}^d: \|l-j\|_1\leq k\}$ and call $\Lambda_k = \Lambda_k(\mathbf{0})$. A \emph{box of size $k$} for some $k \in \mathbb N$ is a set of the form $\Lambda_k(j) $ for some $k \ge 1, j \in k\mathbb Z^d$. For $S \subset \mathbb Z^d$, we call $\B_k(S)$ the set of boxes of size $k$ which is also a subset of $S$. We write $\B_k = \B_k(\mathbb Z^d)$.  Let $\phi_{p,\Lambda_k}^{1}$ and $\phi_{p,\Lambda_k}^{0}$ denote $\phi_{p,\Lambda_k}^{\xi}$ for $\xi=\{\partial\Lambda_k\}$ (wired boundary condition) and $\xi=\{\{l\}:l\in \Lambda_k\}$ (free boundary condition) respectively. Finally we drop the dependence on $d$ in the notation $\Lambda_{n,d}$ and simply refer to it as $\Lambda_n$ in the rest of the proof.
				
		In the above notation (when $\bQ=\bQ(\Lambda_{n,d})=\bQ(\Lambda_n)$ with $\bQ_{ij}=\mathbf{1}(\|i-j\|_1=1)$), we have by the Edward-Sokal coupling (part 3 of Lemma \ref{lemma:random_cluster_model})
		\be 
		\mathrm{Cov}^+_{\beta,\bQ,\mathbf{0}}\left(X_i,X_j\right)&=\phi_{p,\Lambda_n}^{1}(i \leftrightarrow j)-\phi_{p,\Lambda_n}^{1}(i \leftrightarrow \partial\Lambda_n)\phi_{p,\Lambda_n}^{1}(j \leftrightarrow \partial\Lambda_n). \label{eq:cov+}
		\ee 
			We point out although this quantity will turn out to be small as the distance between $i$ and $j$ becomes large, $\E^+_{\beta,\bQ,\mathbf{0}}(X_iX_j)$ just by itself is not small. Indeed, low temperature in the Ising corresponds to the supercritical phase of the random cluster model and hence $\E^+_{\beta,\bQ,\mathbf{0}}(X_iX_j)$ has a uniform lower bound irrespective of the distance between $i$ and $j$. For $d=2$, \eqref{eq:cov+} is an easy consequence of Kramer-Wannier's duality, whose consequence is that the planar dual configuration of the supercritical wired random-cluster measure is the free subcritical random-cluster measure. Therefore, we focus on the case  $d \ge 3$.

		
		For $k=(k_1,\ldots,k_d),j=(j_1,\ldots,j_d) \in \Z^d$, a generalized box $\Lambda_{k} (j) = \{(l_1,\ldots,l_d) \in \Z^d : |l_t  - j_t| \le k_t, t=1,\ldots,d\}$. A generalized box of order $k$ is such that all $k_i \in [k,2k]$.
		Following \cite{duminil2018exponential}, given an $\omega \in \{0,1\}^{E(\Lambda_n)}$, we call a generalized box $S_k$ of order $k$ \emph{good}  if 
		\begin{itemize}
			\item there exists a cluster in $\omega|_{S_{k}}$ touching all the $2d$ sides,
			\item every open path of length at least $k/2$ intersects this cluster.
		\end{itemize}
		In the first item above, $\omega|_{S_{k}}$ denotes the restriction of $\omega$ to $S_k$.
		It follows from a combination of results of \cite{pisz96,B05} that for $p>p_c(d)$.
		\begin{equation}
		\sup_\xi \phi^\xi_{p,S_{k}} (S_k \text{ is good }) \ge 1-e^{-ck}. \label{eq:block_good}
		\end{equation}
		(For ease of reference, let us point out that in \cite{pisz96}, the above result is the statement of (3.7) (with $g(n) = n/2$ there) but with $p>\hat p_1$ and $\alpha =1$ where $\hat p_1$ defined in (3.5) there and it is proved in \cite{B05} that $\hat p_1 = p_c$.) The inequality \eqref{eq:block_good} immediately implies the following useful lemma.
		
			\begin{lemma}\label{lem:going_in}
			Assume that $2\|j - \partial \Lambda_n\|_1 < k$. Then there exists a $c>0$ such that for all $k \ge 1, n \ge 1$,
			$$
			\phi^1_{p,\Lambda_n} (j \leftrightarrow \partial \Lambda_k (j), j \not \leftrightarrow \partial \Lambda_n \text{ inside } \Lambda_k (j)) \le e^{-ck}
			$$
		\end{lemma}
		\begin{proof}
			Notice that at least one of the boundaries of $\Lambda_{k}(j) \cap \Lambda_n$ must have all the vertices belonging to $\partial \Lambda_n$ by assumption. Hence the event in question can only happen if the box $\Lambda_k \cap \Lambda_n$ is not good (since a path joining $j$ to the boundary of $\Lambda_k(j)$ has length at least $k$ and hence it must belong to the cluster which connects all the $2d$ boundary faces) which has probability at most $e^{-ck}$ using \eqref{eq:block_good}.
		\end{proof}

		 We also introduce the natural partial order for configurations in $\{0,1\}^{E(\Lambda_n)}$ by writing $\omega^1 \le \omega^2$ if $\omega^1_e \le \omega^2_e$ for each edge $e \in E(\Lambda_n)$. Throughout the rest of this section, we fix $p > p_c(d)$.
		\begin{lemma}[Exponential mixing]\label{lem:coupling_finite}
			There exists a $k_0 \ge 1, c>0$ such that for all $n \ge 1, k \ge 10 k_0$ the following holds. Let $S_1 ,S_2$ be a unions of boxes of size $k_0$ in $\Lambda_n$ with $\|S_1 - S_2\|_1 > k$. Let $\phi^{1}_{p,\Lambda_n, S_1}$ be the random cluster measure in $\Lambda_n \setminus S_1$ with wired boundary condition and $\phi^{\xi}_{p,\Lambda_n, S_1}$ be the same but with boundary condition $\xi$ on $\partial S_1$ and wired on $\partial \Lambda_n$. Then there exists a coupling $\Phi^{1,\xi}_{p,\Lambda_n} $  on pairs of configurations $(\omega_1,\omega_\xi)$ with $\phi^{1}_{p,\Lambda_n, S_1}$ being the first marginal and $\phi^{\xi}_{p,\Lambda_n, S_1}$ being the second marginal such that $\omega^1 \ge \omega^\xi$ and
			$$
			\Phi^{1,\xi}_{p,\Lambda_n}(\omega_1 |_{S_2}  = \omega_\xi|_{S_2} ) \ge1- e^{ - ck}
			$$
		\end{lemma}
		\begin{proof}
			The proof is the same as Theorem 1.3 of \cite{duminil2018exponential} but unfortunately does not follow from the statement, so we recreate the proof here. Proposition 1.4 of \cite{duminil2018exponential} states that for all $\varepsilon>0$, we can choose a $k$ large enough so that 
			\begin{equation}
			\sum_{e \in E(\Lambda_k)} (\phi^1_{p,\Lambda_{2k}} (\omega_e =1) -   \phi^0_{p,\Lambda_{2k}}(\omega_e =1) ) \le |\Lambda_k| e^{-c(\log k)^{1+c}} \le \varepsilon. \label{eq:block_close}
			\end{equation}
			Indeed, we can make such a choice as the exponential term beats any polynomial in $k$. We choose a $k_0(\varepsilon)$ satisfying the above and fix it and for simplicity assume that $k_0 $ divides $n$.
			
			The idea is to sample both the measures in $\Lambda_n \setminus S_1$ block by block until we end up with the same boundary condition for both measures in the complement of the blocks sampled so far. To that end, we say that a sampled block $B$ is \emph{very good} if $\omega^\xi|_{B} = \omega^1|_{B} $ and it is good in $\omega^\xi|_{B}$ as defined in \eqref{eq:block_good}. There is a set of boundary blocks for the exploration, called $B_t$, with the following properties. The set $B_0$ is simply the set of blocks in $\B_k$ which intersect the boundary of $\Lambda_n \setminus  S_1$. At every step, we pick a block $D$ in $B_t$ and sample $\omega^\xi \le \omega^1$ given the boundary conditions (the partial order in the samples can be maintained using monotonicity of the measures w.r.t. the boundary conditions (See item (i) of Lemma \ref{lemma:random_cluster_model}). If this sampled block turns out to be very good, we remove it from the set of boundary blocks. If not, we declare the block to be sampled and add the blocks in $D \cap \B_k(\Lambda_n)$ which is not in the set of sampled blocks to $B_t$, and define it to be $B_{t+1}$. We see that if $B_t  = \emptyset$, then the rest of the unsampled domain is only surrounded by very good blocks. It is easy to see that this means that the boundary conditions for $\omega^1$ and $\omega^\xi$ coincide in the rest of the domain, and hence both the samples can be taken to be equal. This defines our required coupling $\Phi^{\xi,1}_{p,\Lambda_n}$. Note that the partial order $\omega^\xi \le \omega^1$ is maintained via this coupling.
			
		Observe that \eqref{eq:block_close}, conditioned on sampling the blocks $B_0,B_1,\ldots, B_t$, the probability that $B_{t+1}$ is not very good is at most $\varepsilon$ uniformly over the conditioning.	 Also note that in this coupling $\{\omega^1 |_{S_2}  = \omega^\xi|_{S_2}\}$ does not hold if and only if the blocks $B_t$ intersect $S_2$. This means there must be a set of steps $t_{1} > t_{2} >\ldots>t_s$ with $s>a\frac{k}{k_0}$ for some universal constant $a$ such that at step $t_{i+1}$ we sample a block within distance $3k$ of the block $B_{t_i}$ and is not very good.  Thus overall, this event has probability at most $(\varepsilon)^{a\frac{k}{k_0}}$ which completes the proof.
		\end{proof}
		
		\begin{corollary}\label{cor:coupling_finite}
		There exist $c,c'>0$ such that for all $k \ge 1, n\ge 1$ the following holds. Let $S_1,S_2$ be two boxes of any size inside $\Lambda_n$ with $\|S_1 - S_2\|_1 > k$. Let $A_1$ (resp. $A_2$) be an event which depends only on the states of edges in $S_1$ (resp. $S_2$). Then 
		$$
		|\phi^1_{p,\Lambda_n}(A_1 \cap A_2) -\phi^1_{p,\Lambda_n} (A_1) \phi^1_{p,\Lambda_n}(A_2)|  \le ce^{-c'k}.
		$$
		\end{corollary}
\begin{proof}
Suppose $k \ge 20k_0$ where $k_0$ is as in \ref{lem:coupling_finite}.
In this case, without loss of generality, suppose each of the boxes is a union of boxes of size $k_0 \ge 1$ as in Lemma \ref{lem:coupling_finite} (otherwise take a minimal union of such boxes containing them).   
It follows from the exponential mixing Lemma~\ref{lem:coupling_finite} that for any two configurations $\xi,\xi'$ of the states of edges in $S_2$ 
$$
\phi^1_{p,\Lambda_n} (A_1 | \omega|_{S_2} = \xi )  - \phi^1_{p,\Lambda_n} (A_1 | \omega|_{S_2} = \xi ') \le ce^{-c'k}. 
$$
Averaging over all possible $\xi'$, we obtain
$$
\phi^1_{p,\Lambda_n} (A_1 | \omega|_{S_2} = \xi )  - \phi^1_{p,\Lambda_n} (A_1 ) \le ce^{-c'k}. 
$$
Averaging over all $\xi$ in $A_2$, the corollary is immediate. 

If $k \le 20k_0$, we can trivially bound the left hand side by $2$ and increase $c$ to satisfy the inequality. 
\end{proof}

		We are now ready to prove Theorem~\ref{thm:correlation_decay} for the low temperature regime. We split our argument into several cases depending on where $i$ and $j$ are relative to the boundary of $\Lambda_n$.
		
		\paragraph{Case 1. $\|i - j\|_1 \le 100(\|j-\partial \Lambda_n\|_1 \wedge \|i-\partial \Lambda_n\|_1)$}
		Note that there exist $c,c'>0$ such that for all such $i,j,n$,
		\be
		\ &\phi_{p,\Lambda_n}^1(i \leftrightarrow j) \le \phi_{p,\Lambda_n}^1(i \leftrightarrow \partial \Lambda_{\|i-j\|_1/200}(i), j \leftrightarrow \partial \Lambda_{\|i-j\|_1/200}(j)  ),  \\
		\ &\phi_{p,\Lambda_n}^1(i \leftrightarrow \partial \Lambda_n) \phi_{p,\Lambda_n}^1(j\leftrightarrow \partial \Lambda_n) \ge \phi_{p}(i \leftrightarrow \partial \Lambda_n) \phi_{p}(j\leftrightarrow \partial \Lambda_n),\\
		\ & \left\vert\begin{array}{c}\phi_{p,\Lambda_n}^1(i \leftrightarrow \partial \Lambda_{\|i-j\|_1/200}(i), j \leftrightarrow \partial \Lambda_{\|i-j\|_1/200}(j)  )  \\ - \phi_{p} (i \leftrightarrow \partial \Lambda_{\|i-j\|_1/200}(i), j \leftrightarrow \partial \Lambda_{\|i-j\|_1/200}(j) )\end{array}  \right\vert \le ce^{-c'\|i-j\|_1}.
		\ee
		The first inequality above is simply inclusion.
		The second inequality above follows from monotonicity of boundary conditions (item 1 of Lemma \ref{lemma:random_cluster_model}).
		The third inequality above follows from exponential mixing of the random cluster measure proved in Corollary~\ref{cor:coupling_finite}. To elaborate more on the third inequality, let $x$ be a lattice point closest to the midpoint joining $i$ and $j$. Observe that the event in question is measurable with respect to the FK-Ising configuration inside $\Lambda_{101\|i-j\|_1/200}(x)$ and the boundary of $\Lambda_n$ is at least $\|i-j\|_1/200$ away from $\Lambda_{101\|i-j\|_1/200}(x)$, hence applying Corollary~\ref{cor:coupling_finite}, the third inequality is immediate. 
		
		\begin{figure}[h]
		\centering
		\includegraphics{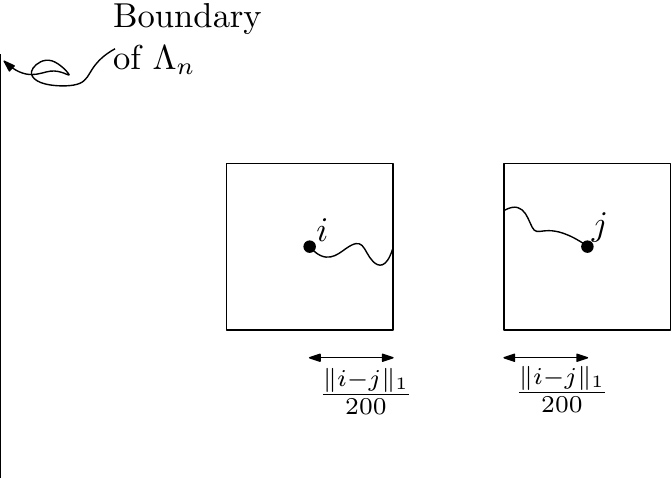}
		\caption{Case 1 (Not drawn to scale). The idea is to show that both $\phi_{p,\Lambda_n}^1(i \leftrightarrow \partial \Lambda_n) \phi_{p,\Lambda_n}^1(j\leftrightarrow \partial \Lambda_n)$ and $\phi_{p,\Lambda_n}^1(i \leftrightarrow j)$ are within  $e^{-c\|i-j\|_1}$ of the above illustrated event. By exponential mixing the above connections from $i$ and $j$ are exponentially close (in $\|i-j\|_1$) to being independent.}\label{F:lt1}
		\end{figure}

		Now it follows from eq. (1.10) in \cite{duminil2018exponential}  that 
		\be
		\left\vert\phi_{p}(i \leftrightarrow \partial \Lambda_n) \phi_{p}(j\leftrightarrow \partial \Lambda_n) - \phi_{p}(i \leftrightarrow \partial \Lambda_{\|i-j\|_1/200}(i)) \phi_{p}(j \leftrightarrow \partial \Lambda_{\|i-j\|_1/200}(j)  )\right\vert \le  ce^{-c' \|i-j\|_1} .
		\ee

		by our choice of $i,j$ in this case.
		Also from Corollary~\ref{cor:coupling_finite} (see also discussion in \cite[Section 1.4]{duminil2018exponential}), 
		\be
 \left\vert \begin{array}{c}\phi_{p} (i \leftrightarrow \partial \Lambda_{\|i-j\|_1/200}(i), j \leftrightarrow \partial \Lambda_{\|i-j\|_1/200}(j) )  \\ -  \phi_{p}(i \leftrightarrow \partial \Lambda_{\|i-j\|_1/200}(i)) \phi_{p}(j \leftrightarrow \partial \Lambda_{\|i-j\|_1/200}(j) ) \end{array}\right\vert  \le ce^{-c'\|i-j\|_1}.
		\ee
		
		Combining all of the above, we have the exponential decay of \eqref{eq:cov+}.
		\medskip
		
		\paragraph{Case 2. $\|i - j\|_1 > 100\|j-\partial \Lambda_n\|_1 $ and $j$ is closer to $\partial \Lambda_n$ than $i$} 
		\begin{figure}[h]
		\centering
		\includegraphics[scale  = 0.72]{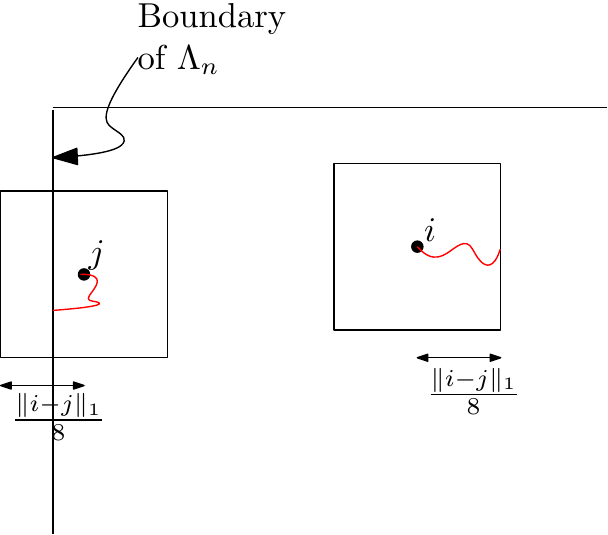}\qquad \qquad \qquad \qquad
		\includegraphics[scale = 0.9]{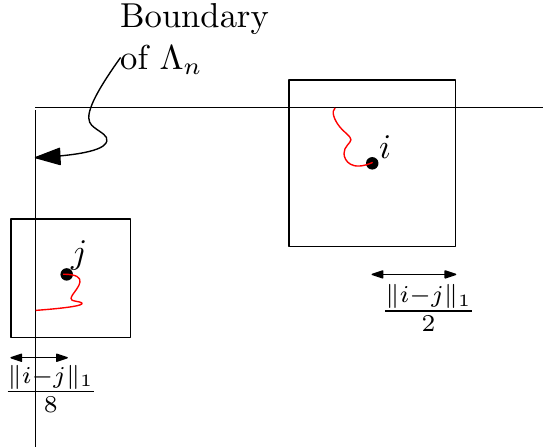}
		\caption{Case 2: (Not drawn to scale) Illustration of the events $A$ and $B$ for Case 2 and $\|i-j\|_1/4 \leq\|i - \partial \Lambda_n\|_1$ (left) and for Case 2 and $\|i-j\|_1/4 >\|i - \partial \Lambda_n\|_1$ (right)}.\label{F:lt2}
		\end{figure}

		Note that
		$$
		\phi^1_{p,\Lambda_n} (j \leftrightarrow \partial \Lambda_n, j \not \leftrightarrow \partial \Lambda_n \text{ inside } \Lambda_{\|i-j\|_1/4  }(j)) \le e^{-c\|i-j\|_1}.
		$$
		simply because if the above event occurs then $j$ must be connected to $\partial \Lambda_{\|i-j\|_1/4}(j)$ and we can employ Lemma \ref{lem:going_in} (with $k = \|i-j\|_1/4$). Therefore,
		\be
		|\phi^1_{p,\Lambda_n}(j \leftrightarrow \partial \Lambda_n) - \phi^1_{p,\Lambda_n}(j \leftrightarrow \partial \Lambda_n \text{ inside } \Lambda_{\|i-j\|_1/4}(j)) |\le e^{-c\|i-j\|_1}.
		\ee
		Now first assume that $\|i - j\|_1/4 \leq\|i - \partial \Lambda_n\|_1$ (see Figure \ref{F:lt2}). Then firstly
	\be
	\ & 	\phi^1_{p,\Lambda_n} (i \leftrightarrow j) \\
	&\le \phi^1_{p,\Lambda_n} (i \leftrightarrow \partial \Lambda_{\|i - j\|_1/8}(i) , j \leftrightarrow\partial \Lambda_{\|i-j\|_1/8}(j) ) \\
		&\le \phi^1_{p,\Lambda_n} (i \leftrightarrow \partial \Lambda_{\|i-j\|_1/8}(i) , j \leftrightarrow \partial \Lambda_{n} \text{ inside } \Lambda_{\|i-j\|_1/8 }(j) ) + e^{-c\|i-j\|_1}.
\ee
		using Lemma~\ref{lem:going_in}. Also note that
		$$
		| \phi_{p,\Lambda_n}^1 (i \leftrightarrow \partial \Lambda_n) - \phi^1_{p,\Lambda_n}(i \leftrightarrow \partial \Lambda_{\|i-j\|_1/8}(i))   | \le e^{-c\|i-j\|_1}
		$$
		We now argue why the above equation is true. Firstly, the above equation is immediate for $\phi_p$ in place of $\phi_{p,\Lambda_n}^1$ using eq.(1.10) of \cite{duminil2018exponential}. Next, we use Corollary~\ref{cor:coupling_finite} to say that $\phi_p$ and $\phi^1_{p,\Lambda_n}$ probabilities of $\{i \leftrightarrow \partial \Lambda_{\|i-j\|_1/8}(i)\}$ are exponentially close in $\|i-j\|_1$ since $\|i-j\|_1/8 < \|i-\partial \Lambda_n\|_1/2$ from our assumption.  Therefore we need to bound, 
		$$
		| \phi^1_{p,\Lambda_n} (A \cap B )   - \phi^1_{p,\Lambda_n} (A)\phi^1_{p,\Lambda_n} (B)|.
		$$
		where $A  = \{ i \leftrightarrow \partial \Lambda_{\|i-j\|_1/8}(i) \}$ and $B =  \{j \leftrightarrow \partial \Lambda_{n} \text{ inside } \Lambda_{\|i-j\|_1/8 }\}$. It now follows from Corollary \ref{cor:coupling_finite} that the above quantity is bounded by $ce^{-c'\|i-j\|_1}$ since the events $A$ and $B$ are measurable with respect to configuration inside boxes which are $c\|i-j\|_1$ apart for some $c>0$.
		
		Finally assume $\|i-j\|_1/4 >\|i - \partial \Lambda_n\|_1$((see Figure \ref{F:lt2})). In this case, again using Lemma \ref{lem:going_in} (with $k  = \|i-j\|_1/2$ for $i$ and $k = \|i-j\|_1/8$ for $j$), we see that
	\be
		\ & \phi^1_{p,\Lambda_n} (i \leftrightarrow j) \\
		&\le \phi^1_{p,\Lambda_n} (i \leftrightarrow \partial \Lambda_{n} \text{ inside } \Lambda_{\|i-j\|_1/2 }(i) , j \leftrightarrow \partial \Lambda_{n} \text{ inside } \Lambda_{\|i-j\|_1/8 }(j) ) + e^{-c\|i-j\|_1}
		\ee
		and also 
		\be
		&|\phi^1_{p,\Lambda_n} (i \leftrightarrow \partial \Lambda_n ) - \phi^1_{p,\Lambda_n} (i \leftrightarrow \partial \Lambda_{n} \text{ inside } \Lambda_{\|i-j\|_1/2 }(i)) |  \le e^{-c\|i-j\|_1}\\
		&| \phi^1_{p,\Lambda_n} (j \leftrightarrow \partial \Lambda_n ) - \phi^1_{p,\Lambda_n} ( j \leftrightarrow \partial \Lambda_{n} \text{ inside } \Lambda_{\|i-j\|_1/8 }(j))|  \le e^{-c\|i-j\|_1}.
		\ee
		
		Therefore we need to bound, 
		$$
		| \phi^1_{p,\Lambda_n} (A \cap B )   - \phi^1_{p,\Lambda_n} (A)\phi^1_{p,\Lambda_n} (B)|.
		$$
		where $A  = \{  i \leftrightarrow \partial \Lambda_{n} \text{ inside } \Lambda_{\|i-j\|_1/2 }(i)) \}$ and $B =  \{j \leftrightarrow \partial \Lambda_{n} \text{ inside } \Lambda_{\|i-j\|_1/8 }(j))\}$. It now follows from Corollary~\ref{cor:coupling_finite} that the above quantity is bounded by $ce^{-c'\|i-j\|_1}$ as again $A$ and $B$ are measurable with respect to the configuration inside boxes which are $c\|i-j\|_1$ apart. This completes the proof.

		\subsection{Proof of Theorem \ref{theorem:estimation}}
		We first prove that under the conditions of the theorem, for any fixed $\beta^{*}$
		\be 
		\sup_{(\beta^{*},h)\in \mathbb{R}^+\times \mathbb{R}}\E_{\beta,\bQ,\bmu(h)}\left(\hat{h}-h\right)^2\geq \frac{c}{\mathrm{Var}_{\beta,\bQ,\bmu(0)}\left(\sum_{i=1}^n X_i\right)}.
		\ee
		
		To this we note by Le-Cam's two point argument it is enough to consider the hypothesis testing problem \eqref{eqn:hypo_sparse} with $\beta=\beta^*$, $s=n$, and $A=t/\sqrt{\mathrm{Var}_{\beta^*,\bQ,\bmu(0)}\left(\sum_{i=1}^n X_i\right)}$ and show that there exists a $t_0$ such that for $t<t_0$, then no tests are asymptotically powerful.
		
		A standard second moment approach shows that it is enough to control the $$\E_{\beta^*,\bQ,\bmu(0)}\left(L^2\right) \quad \text{with} \quad L=\frac{Z(\beta^*,\bQ,\bmu(0))}{Z(\beta^*,\bQ,\bmu(A))}\exp\left(A\sum_{i=1}^n X_i\right).$$
		By a direct calculation we have by a two term Taylor expansion and thereafter using the properties of exponential family
		\be
		\ & \E_{\beta^*,\bQ,\bmu(0)}\left(L^2\right)\\
		&=\frac{Z(\beta^*,\bQ,\bmu(0))Z(\beta^*,\bQ,\bmu(2A))}{Z^2(\beta^*,\bQ,\bmu(A))}\\
		&=\exp\left(2A^2\mathrm{Var}_{\beta^*,\bQ,\bmu(\eta_1)}\left(\sum_{i=1}^n X_i\right)-A^2\mathrm{Var}_{\beta^*,\bQ,\bmu(\eta_2)}\left(\sum_{i=1}^n X_i\right)\right)\\
		&\leq \exp\left(3A^2\mathrm{Var}_{\beta^*,\bQ,\bmu(0)}\left(\sum_{i=1}^n X_i\right)\right)\leq \exp\left(3t^2\right)
		\ee
		Above the third last line follows with $\eta_1\in [0,2A]$ and $\eta_2\in [0,A]$, the second last line uses the GHS inequality (Lemma \ref{lemma:GHS}), and the last line uses the definition of $A$. Consequently, $\E_{\beta^*,\bQ,\bmu(0)}\left(L^2\right)$ can be made less than $1+\epsilon$ for any $\epsilon>0$ by choosing $t<\sqrt{\log(1+\epsilon)/3}$ and the proof follows.

		We now prove that under the conditions of the theorem, for any fixed $\beta^{*}>0$ and take $t^*$ small enough such that $\frac{t^*}{\sqrt{\sum_{ij}\bQ_{ij}^2}}<\delta<\beta^*$
		\be 
		\sup_{\beta\in (\beta^{*}-\delta,\beta^*+\delta) }\E_{\beta,\bQ,\bmu(0)}\left(\hat{\beta}-\beta\right)^2\geq \frac{c}{\sum_{ij}\bQ_{ij}^2}.
		\ee
		
		Once again by Le-Cam's two point argument it is enough to consider the hypothesis testing problem $H_0:\beta=\beta^*$ and $H_1:\beta=\beta(t):=\beta^*+t/\sqrt{\sum_{ij}\bQ_{ij}^2}$ (with $t<t^*$) \eqref{eqn:hypo_sparse} and show that there exists a $t_0<t^*$ such that for $t<t_0$, then no tests are asymptotically powerful. 
		
		Once again standard second moment approach shows that it is enough to control the $$\E_{\beta^*,\bQ,\bmu(0)}\left(L^2\right) \quad \text{with} \quad L=\frac{Z(\beta^*,\bQ,\bmu(0))}{Z(\beta(t),\bQ,\bmu(0))}\exp\left((\beta(t)-\beta^*)\frac{\bX^T\bQ\bX}{2}\right).$$
		Letting $\beta(t)=\beta^*+t/\sqrt{\sum_{ij}\bQ_{ij}^2}$ and $r_t=\beta(t)-\beta^*$, by a direct calculation we have by a two term Taylor expansion and thereafter using the properties of exponential family
		\be
		\ & \E_{\beta^*,\bQ,\bmu(0)}\left(L^2\right)\\
		&=\frac{Z(\beta(0),\bQ,\bmu(0))Z(\beta(2t),\bQ,\bmu(0))}{Z^2(\beta(t),\bQ,\bmu(0))}\\
		&=\exp\left(2r_t^2\frac{\partial^2}{\partial\beta^2}\log Z(\beta,\bQ,\bmu(0))_{\beta=\beta_1}-r_t^2\frac{\partial^2}{\partial\beta^2}\log Z(\beta,\bQ,\bmu(0))_{\beta=\beta_2}\right)
		\ee
		where $\beta_1\in [\beta(0),\beta(2t)]$ and $\beta_2\in [\beta(0),\beta(t)]$. Now note that
		\be 
		\frac{\partial^2}{\partial\beta^2}\log Z(\beta,\bQ,\bmu(0))&=\mathrm{Var}_{\beta,\bQ,\bmu(0)}\left(\frac{\bX^T\bQ\bX}{2}\right).
		\ee
Hence
		\be 
		\E_{\beta^*,\bQ,\bmu(0)}\left(L^2\right)&\leq \exp\left(2r_t^2\mathrm{Var}_{\beta_1,\bQ,\bmu(0)}\left(\frac{\bX^T\bQ\bX}{2}\right)\right).
		\ee
		Now, since $\beta_1\leq \beta^*+\delta<(1-\rho)/\|\bQ\|_{\infty\rightarrow \infty}$ we have by \cite[Theorem 2.1]{gheissari2018concentration} that 
		\be
		\mathrm{Var}_{\beta_1,\bQ,\bmu(0)}\left(\frac{\bX^T\bQ\bX}{2}\right)&\leq C\sum_{ij}\bQ_{ij}^2,
		\ee
		for some constant $C>0$ depending on $\beta^*,\delta,\rho$. Consequently, using the definition of $r_t^2$ we have
		\be 
		\E_{\beta^*,\bQ,\bmu(0)}\left(L^2\right)&\leq \exp\left(2Ct^2\right),
		\ee
		and the proof follows as before.



		\bibliographystyle{imsart-nameyear}
		\bibliography{biblio_ising_lattices}

	\end{document}